\documentclass[11pt]{amsart}
\usepackage[centering]{geometry}
\pdfoutput=1
\usepackage{amsmath,amsfonts,amsthm,graphicx}
\usepackage{hyperref}
\usepackage{graphicx,psfrag}  
\usepackage[psamsfonts]{amssymb}
\usepackage{amscd}
\usepackage[all,cmtip]{xy}
\usepackage{mathtools}
\usepackage{setspace}
\title{Simplicity of the Lyapunov spectrum for classes of Anosov flows}
\author{Daniel Mitsutani}
\address[Mitsutani]{Department of Mathematics, the University of Chicago, Chicago, IL, USA, 60637}
\email{mitsutani@math.uchicago.edu}

\date{\today}
\theoremstyle{plain}

\newtheorem{theorem}{Theorem}[section]
\newtheorem{proposition}[theorem]{Proposition}
\newtheorem{lemma}[theorem]{Lemma}
\newtheorem{corollary}[theorem]{Corollary}

\newtheorem{definition}[theorem]{Definition}
\newtheorem{remark}[theorem]{Remark}

\def\eps{\varepsilon}

\def\title{\em}

\def\bar{\overline}

\def\cW{\mathcal{W}}

\def\cE{\mathcal{E}}
\def\cV{\mathcal{V}}

\def\cU{\mathcal{U}}

\def\cA{\mathcal{A}}

\def\cM{\mathcal{M}}
\def\G{\mathcal{G}}
\def\cC{\mathcal{C}}
\def\Or{\mathcal{O}}

\def\P{\mathcal{P}}

\def\cP{\mathcal{P}}

\def\transverse{\,\raise2pt\hbox to1em{\hfil$\top$\hfil}\hskip -1em \hbox
to1em{\hfil$\cap$\hfil}\,} 
\newcommand\Z{\mathbb{Z}}
\newcommand\R{\mathbb R}
\newcommand\C{\mathbb{C}}
\newcommand\exps{\vec{\lambda}}

\newcommand\N{{\mathbb N}}

\newlength{\figboxwidth} 

\begin{document}

\begin{abstract} We prove that in a $C^1$-open and $C^k$-dense set of some classes of $C^k$ Anosov flows all Lyapunov exponents have multiplicity 1 with respect to appropriate measures. The classes are geodesic flows with equilibrium states of Hölder-continuous potentials, volume-preserving flows, and all fiber-bunched Anosov flows with equilibrium states of Hölder-continuous potentials. In the proof, we use and prove perturbative results for jets of flows to modify eigenvalues of certain Poincaré maps and, using a Markov partition, apply the simplicity criterion of Avila and Viana \cite{av}.
\end{abstract}

\maketitle

\section{Introduction}

The existence of a positive Lyapunov exponent and more generally the multiplicity of the Lyapunov exponents of a system are of essential interest due to their relation to other dynamical invariants and the geometry of the associated dynamical foliations. In this paper, we seek to address the question of how often simplicity (i.e. all exponents of multiplicity 1) of Lyapunov spectrum arises for some classes of hyperbolic flows.

In \cite{bv}, Bonatti and Viana first established a criterion for simplicity of Lyapunov spectrum of a cocycle over a discrete symbolic base which holds in great generality with respect to a large class of measures. Applying a Markov partition construction, the authors also extend the results to cocycles over  hyperbolic maps, which naturally leads to the question of whether the criterion generically holds for the derivative cocycle in the space of diffeomorphisms. Indeed, without any further restrictions, the arguments in \cite{bv} can be modified without much difficulty to show that such a result would be possible, for appropriate choices of measures. 

Here we consider the question of genericity of simple spectrum in the continuous-time setting -- in particular in more restrictive classes (geodesic flows, conservative flows, etc.) of Anosov flows, which presents significant differences relative to the discrete-time scenario. We establish a method of constructing appropriate perturbations of the Lyapunov spectrum by perturbing the 1-jet of an appropriate Poincaré map within a given class. 

We apply it in different settings to obtain the following results. Let $X$ be a smooth closed manifold; precise definitions of the other terms below are given in Section \ref{sec:prel}:

\begin{theorem}[Geodesic flows]\label{main} For $3 \leq k \leq \infty$ we denote by $\mathcal{G}^k$ the set of $C^{k}$-Riemannian metrics on $X$ with sectional curvatures $1 \leq -K < 4$. 
	
	There exists a $C^2$-open and $C^k$-dense set in $\G^{k}$ of metrics such that with respect to the equilibrium state of any Hölder potential (e.g. Liouville measure, m.m.e.) the derivative cocycle of the geodesic flow has simple Lyapunov spectrum, i.e., all its Lyapunov exponents have multiplicity one. 
\end{theorem}

\begin{theorem}[Conservative flows] \label{volmain} For a fixed smooth volume $m$ and for $2 \leq k \leq \infty$ let $\mathfrak{X}^k_{m}(X)$ be the set of divergence-free (with respect to $m$) $C^k$ vector fields on $M$ which generate (strictly) $\frac{1}{2}$-bunched Anosov flows. 
	
	Then flows in a $C^1$-open and $C^k$-dense set of $\mathfrak{X}^k_{m}(X)$ have simple Lyapunov spectrum with respect to $m$.
	
\end{theorem}

\begin{theorem}[All flows] \label{generalmain} For $2 \leq k \leq \infty$ let $\mathfrak{X}^k_A(X)$ be the set of  $C^k$ vector fields on $M$ which generate (strictly) $\frac{1}{2}$-bunched Anosov flows. 
	
	Then flows in a $C^1$-open and $C^k$-dense set of $\mathfrak{X}^k_A(X)$ have simple Lyapunov spectrum with respect to the equilibrium state of any Hölder potential (e.g. SRB measure, m.m.e.).
	
\end{theorem}

As indicated before, the proofs are accomplished by constructing a discrete symbolic system via a Markov partition to apply a simplicity criterion of Avila and Viana \cite{av}, which is itself an improvement of the criterion of Bonatti and Viana \cite{bv} aforementioned. In each class, we prove or use a previously established perturbational result to obtain density in the theorems above.

One main difficulty particular to the setting of $\R$-cocycles which was already present in \cite{bv} arises in attempting to perturb the norms of pairs of complex eigenvalues generically. In \cite{bv}, through the introduction of rotation numbers which vary continuously with the perturbation for orbits near a periodic point, a small rotation on a periodic orbit is propagated to an arbitrarily large one for a homoclinic point, which can then be made to have real eigenvalues. 

While such rotation numbers are well-defined for the particular perturbation of the cocycle introduced in \cite{bv}, a general construction which allows for perturbations of the base system has only been introduced recently in \cite{gourm}. However, the constructions in \cite{gourm} do not apply directly to flows, and so we introduce new ideas to control the eigenvalues of the cocycle in the continuous-time setting. 

Since the class of geodesic flows is the substantially more difficult case, we carry out the proof of Theorem \ref{main} in detail, and in Section \ref{sec:volpres} we prove the analogous results needed for Theorem \ref{volmain}.

\subsection{Outline} In Section \ref{sec:prel} we give the necessary background for the later sections; we summarize the main results of \cite{klta} and \cite{av} and introduce rotation numbers. For a more basic introduction to Lyapunov exponents and cocycles we refer the reader to \cite{v} and for background on geodesic flows \cite{pa}. In Sections \ref{sec: pt} and \ref{sec: proof} we specialize to the setting of the geodesic flows, giving the main arguments to prove of Theorem \ref{main}. Finally, in Section \ref{sec:volpres} we prove a perturbational result for the volume-preserving class, which by direct adaptation of the arguments of the previous sections proves Theorem \ref{volmain} and Theorem \ref{generalmain}.

\subsection{Acknowledgements} I would like to thank Amie Wilkinson for all her suggestions and continued guidance in the process of research leading up to this paper, and also for her help in reviewing the text. 

\section{Preliminaries} \label{sec:prel}

\subsection{Lyapunov exponents and Simplicity of Spectrum}  \label{ssec:bvi} Here we collect and fix the definitions and background results used in later sections. For a continuous flow $\Phi^t: X \to X$ on a compact metric space $X$ preserving an ergodic measure $\mu$, a continuous linear cocycle over $\Phi$ on a  linear bundle $\pi: \mathcal{E}\to X$ is a continuous map ${\cA}: \R \times \cE \to\cE$ such that the maps $$A^t_{\pi(v)} := \pi \circ {\cA}(t, \pi(v)): \cE_{\pi(v)} \to \cE_{\Phi^t(\pi(v))}$$
 are linear isomorphisms of the fibers and $\Phi^t \circ \pi = \pi \circ \cA^t$, where $\cA^t := \cA(t, \cdot)$.

Suppose $\log^+ \|A^t(x)\| \in L^1(X, \mu)$ for all $t \in \R$. For some fixed choice of norm $\|\cdot \|$ on the fibers, the fundamental result describing asymptotic growth of vectors under $\cA$ is Oseledets' theorem: there exists a set of numbers $\lambda_1, ..., \lambda_n \in \R$, with $\lambda_i \neq \lambda_j$ for $i\neq j$, a measurable splitting $\cE = \cE^1 \oplus \dots \oplus \cE^n$ and a set of full measure $Y \subseteq X$ such that for all $x \in Y$ and $t \in \R$ we have $A_x^t \cE^i_x = \cE^i_{\Phi^t(x)}$ and moreover for $v \in \cE^i_x$: 
	$$\lim_{t \to \pm \infty} \frac{1}{t} \log ||A_x^t v|| = \lambda_i.$$ 
	
	The numbers $\lambda_i$ are the \textit{Lyapunov exponents} of $\cA$ with respect to $\mu$.

 When all bundles $\cE^i$ are 1-dimensional, $\cA$ is said to have \textit{simple Lyapunov spectrum} with respect to $\mu$. When $X$ is a smooth manifold and $\Phi^t$ is $C^1$, the \textit{dynamical cocycle} on $\cE = TX$  is the derivative map $D\Phi^t$ of the flow, we often refer to its Lyapunov exponents as the Lyapunov exponents of $\Phi$ with respect to $\mu$ -- similarly, we say $\Phi$ has simple Lyapunov spectrum when the dynamical cocycle does.
 
 The definitions above hold in the discrete-time setting  of \cite{av}, with appropriate modifications, where the criterion for simplicity of Lyapunov spectrum we need is proved -- following their notation, we let $\hat{f}$ be the shift on the space $\hat{\Sigma} = \N^\Z$ and $\cA$ be a measurable cocycle on $\hat{\Sigma} \times \R^d$ over $\hat{f}$, which alternatively can be equivalently described by some measurable $\hat{A}: \hat{\Sigma}_T \to GL(d,\R)$ when the bundle is trivial, a harmless assumption since all bundles we consider are measurably trivializable. 
 
The theorems of Avila and Viana all require the additional bunching assumption: 
 
 \begin{definition}[Domination/Holonomies]\label{dom}
 	$\hat{A}$ is dominated if there exists a distance $d$ in $\hat{\Sigma}$ and constants $\theta < 1$ and $\nu \in (0, 1]$ such that, up to replacing $\hat{A}$ by some power $\hat{A}^N$:
 	\begin{enumerate}
 		\item [(1)] $d(\hat{f}(\hat{x}), \hat{f}(\hat{y})) \leq \theta d(\hat{x}, \hat{y})$ and $d(\hat{f}^{-1}(\hat{x}), 
 		\hat{f}^{-1}(\hat{y})) \leq \theta d(\hat{x}, \hat{y})$ for every $\hat{y} \in W^s_{loc}(\hat{x})$ and $\hat{z} \in W^u_{loc}(\hat{x})$ 
 		\item [(2)] The map $\hat{x} \mapsto \hat{A}(\hat{x})$ is $\nu$-Hölder continuous and $\|\hat{A}(\hat{x})\|\|\hat{A}^{-1}(\hat{x})\|\theta^\nu < 1$ for every $\hat{x} \in \hat{\Sigma}$.
 		
 	\end{enumerate}

 	If $\hat{A}$ is either dominated or constant on each cylinder, there exists a family of holonomies $\phi^u_{\hat{x}, \hat{y}}$, i.e., linear isomorphisms of $\R^d$ such that for each $\hat{x}, \hat{y} \in \hat{\Sigma}$ in the same local unstable manifold of $\hat{f}$ there exists $C_1 > 0$ such that: 
 	\begin{enumerate}
 		\item [(1)] $\phi^u_{\hat{x}, \hat{x}} = id$ and $\phi^u_{\hat{x}, \hat{y}} = \phi^u_{\hat{x}, \hat{z}} \circ \phi^u_{\hat{z}, \hat{y}}$,
 		\item [(2)]$\hat{A}(\hat{f}^{-1}(\hat{y})) \circ \phi^u_{\hat{f}^{-1}(x), \hat{f}^{-1}(y)} \circ \hat{A}^{-1}(\hat{x}) = \phi^u_{\hat{x}, \hat{y}},$
 		\item [(3)] $\|\phi^u_{\hat{x}, \hat{y}} - id\|  \leq C_1 d(\hat{x}, \hat{y})^\nu$.
 	\end{enumerate}
 
 There is a family $\phi^s$ of holonomies over stable manifolds satisfying analogous properties.
\end{definition}

For such cocycles, the holonomies allow to propagate the dynamics over single periodic orbits to obtain data on the Lyapunov spectrum of certain measures. Thus, the adaptation of the original pinching and twisting conditions for a monoid of matrices can be adapted to these cocycles as follows:
 
 \begin{definition}[Simple cocycles] \label{deftyp} Suppose $\hat{A} :  \hat{\Sigma} \to GL(d, \R)$ is either dominated or constant on each cylinder of $\hat{\Sigma}$ . We say that $\hat{A}$ is \textit{simple} if there exists a periodic point $\hat{p}$ and a homoclinic point $\hat{z}$ associated to $\hat{p}$ such that:
 	\begin{enumerate}
 		\item [(P)] the eigenvalues of $\hat{A}$ on the orbit of $\hat{p}$ have multiplicity 1 and distinct norms -- let $\omega_j \in \R P^{d-1}$ represent the eigenspaces, for $1 \leq j \leq d$; and
 		\item [(T)]	 $\{\psi_{\hat{p},\hat{z}}(\omega_i): i \in I\} \cup \{\omega_j: j \in J\}$ is linearly independent, for all subsets $I$ and $J$ of
 		${1,...,d}$ with $\# I + \# J \leq d$ where, denoting by $\phi^u$ and $\phi^s$ the holonomies as above,
 		$$\psi_{\hat{p},\hat{z}} =  \phi^s_{\hat{z},\hat{p}} \circ  \phi^u_{\hat{p},\hat{z}}.$$
 	\end{enumerate}
 \end{definition}
 
 An invariant probability measure $\hat{\mu}$ has local product structure if for every cylinder $[0:i]$:
 $$\hat{\mu}|[0:i] = \psi \cdot (\mu^+ \times \mu^-)$$
 where $\psi: [0:i] \to \R$ is continuous and $\mu^+$ and $\mu^-$ are the projections of $\hat{\mu}|[0:i]$ to spaces of one-sided sequences indexed by positive and negative indices respectively. For instance, this property holds for every equilibrium state of $\hat{f}$ associated to a Hölder potential \cite{bow}. 
 
 \begin{theorem} \label{bvmainthm}\cite[Theorem A]{av}
 	If  $\hat{A}$ is a simple cocycle then it has Lyapunov exponents of multiplicity one with respect to any $\hat{\mu}$ with local product structure.
 \end{theorem}

\subsection{Anosov Flows} The continuous-time hyperbolic systems we study are:

\begin{definition}[$C^k$-Anosov Flows] A $C^k$ ($1 \leq k \leq \infty$) flow $\Phi^t: X \to X$ on a smooth manifold $X$ is called \textit{Anosov} if it preserves a splitting $E^u \oplus E^0 \oplus E^s$ of $TX$ such that $E^0$ is the flow direction and there exist $\lambda > 0$ and $C > 1$ such that for all $v \in E^u$ and $u \in E^s$:
	$$||D\Phi^t v|| \geq Ce^{\lambda t} ||v||,\, \hspace{.3cm} ||D\Phi^{-t} u|| \geq Ce^{\lambda t} ||u||.$$
\end{definition}

A significant class of cocycles over Anosov flows related to the theory of partially hyperbolic systems and to the class of dominated cocycles over shift maps is that of \textit{fiber bunched cocycles}, whose expansion and contraction rates are dominated by the base dynamics:

\begin{definition}[Fiber Bunching] \label{fbunchdef} A $\beta$-Hölder continuous cocycle $\mathcal{A}: \mathcal{E} \times \R \to \mathcal{E}$ over an Anosov flow $\Phi^t:X \to X$ is said to be $\alpha$-fiber bunched if $\alpha \leq \beta$ and there exists $T> 0$ such that for all $p \in M$ and $t \geq T$:
	$$\|A^t_p\| \|A_p^{-t}\| \| D\Phi^t|_{E^s}\|^{\alpha} < 1, \hspace{.3cm} \|A^t_p\| \|A_p^{-t}\| \| D\Phi^{-t}|_{E^u}\|^{\alpha} < 1.$$
	
When the cocycles $D\Phi^t|_{E^{i = u,s}}$ themselves satisfy the inequalities above in place of $\cA$, the Anosov flow is said to be $\alpha$-bunched.
\end{definition}

 Fiber bunching is a partial hyperbolicity condition on the projectivization of the fiber bundle, with the fibers composing the center direction and the base system the stable and unstable directions. The strong stable and unstable manifold theorem can be interpreted as defining holonomy maps between the fibers:

\begin{theorem} \cite{kalsad} \label{holsflow}
	Suppose $\mathcal{A}$ is $\beta$-Hölder and fiber bunched over a base system as in Definition \ref{fbunchdef}. Then the cocycle admits \textit{holonomy maps} $h^u$, that is, a  continuous map $h^u: (x,y) \to h^u_{x,y}$, $x\in M$, $y \in W^u_{loc}(x)$, such that:
	\begin{enumerate}
		\item [(1)] $h^u_{x,y}$ is a linear map $\mathcal{E}_x \to \mathcal{E}_y$,
		\item [(2)] $h^u_{x,x} = Id$  and  $h^u_{y,z} \circ  h^u_{x,y} = h^u_{x,z}$,
		\item [(3)] $h^u_{x,y} = (A^t_y)^{-1} \circ h^u_{\Phi^t(x), \Phi^t(y)} \circ A^t_x$ for every $t \in \R$. 
	\end{enumerate}
	
	Moreover, the holonomy maps are unique, and, fixing a system of linear identifications $I_{xy}: \mathcal{E}_x\to \mathcal{E}_y$, see \cite{kalsad},  they satisfy:
	$$\|h^u_{x,y}- I_{x,y} \| \leq C d(x,y)^\beta.$$
	
	Using property (3), one may extend these holonomies for all $y \in W^{cu}(x)$ (as opposed to $W^u_{loc}(x)$), and such holonomies are denoted by $h^{cu}$. 
\end{theorem}

For the case where $\Phi$ is itself $\alpha$-bunched, it is known that \cite{hass} the bunching constant is directly related to the regularity of the Anosov splitting: for $\frac{1}{2}$-bunched Anosov flows, the weak stable and unstable bundles $E^{cu,cs} := E^0 \oplus E^{u,s}$ are of class $C^1$. Thus:

\begin{proposition} \label{bunchingquotient}
	For $\Phi^t: X \to X$ a $\frac{1}{2}$-bunched $C^2$-Anosov flow, the cocycle $\cA^u$ (resp. $\cA^s$) on the bundle $Q^u:= E^{cu}/E^0$ (resp. $Q^s:=E^{cs}/E^0$)  given by the derivative $D\Phi^t$ is $1$-bunched. 
\end{proposition}
\begin{proof} The cocycle $D\Phi^t|_{E^{cu}/E^0}$ is $C^1$ by the regularity of the splitting mentioned above, and by hypothesis $D\Phi^t|_{E^{cu}/E^0}$ satisfies the inequalities in the definition of fiber bunching with $\alpha = 1$. Same for $E^{cs}$
\end{proof}

Finally, we describe the class of measures with respect to which we prove our results. Fix a topologically mixing $C^2$-Anosov flow $\Phi^t$. Let  Let $\rho: X \to \R$ be a Hölder-continuous function, which we refer to as a \textit{potential}. Then an \textit{equilibrium state} $\mu_\rho$ of $\rho$ is an invariant measure satisfying the variational principle:
$$ h_{\mu_\rho}(\Phi) + \int \rho \, d\mu_\rho = \sup_{\mu \in \mathcal{M}_\Phi(X)} h_{\mu}(\Phi) + \int \rho \, d\mu,$$
where $h_\mu(\Phi)$ is the measure-theoretic entropy of $\Phi$ with respect to $\mu$ and $ \mathcal{M}_\Phi(X)$ is the set of invariant measures of $\Phi$. The existence and uniqueness of the equilibrium state $\mu_\rho$, is, in this setting, a foundational result in the theory of the thermodynamical formalism \cite{bow}. 

Important examples of equilibrium states include the case $\rho = 0$, which gives the measure of maximal entropy as the equilibrium state, and $\rho(x) = - \frac{d}{dt} \log J^u(x,t)|_{t=0}$, where $J^u(x,t) = \det {D_x\Phi^t|_{E^u}}$, which gives the SRB measure. Moreover, the product structure property mentioned in the previous section is also a classical result for equilibrium states proved in \cite{bow}.

\subsection{Rotation Numbers} \label{ssec:rotation} As indicated in the introduction, in order to perturb away complex eigenvalues by a small rotation, one needs the formalism of rotation numbers, which we introduce in complete form here. We roughly follow the discussion in Section 3 of \cite{gourm}.

As a brief introduction, recall that for an orientation preserving homeomorphism of the circle $f: S^1 \to S^1$, the \textit{Poincaré rotation number} $\rho(f) \in S^1  = \R/2\pi$ of $f$ is defined as:
$$ \rho(f) = \lim_{n \to \infty} \frac{\tilde{f}^n(x)-x}{n} \text{ (mod }2\pi),$$ 
for a lift $\tilde{f}: \R \to \R$ of $f$. This limit always exists and is independent of the choice of $x \in \R$ and the lift $\tilde{f}$. For an orientation reversing homeomorphism we define $\rho(f) = 0$. 

The Poincaré rotation number measures, on average, how much an element is rotated by an application of $f$ and is a conjugation invariant, i.e., $\rho(g^{-1}f g) = \rho(f)$, for $g$ also a homeomorphism of $S^1$. In what follows we extend this definition for cocycles on circle bundles.

Throughout this section, let $X$ a compact metric space and $\Phi^t$ a continuous flow on $X$. Let $\cM_\Phi(X)$ be the space of probability measures on $X$ invariant under $\Phi^t$ with the weak-* topology. 

For our purposes, it will suffice to work with trivial bundles $E = X \times S^1$, and a continuous cocycle ${\cA}: \R \times E \to E$ over $\Phi^t$. Then for $(x, \theta) \in E$, the map $t \mapsto A^t_x(\theta)$ is a continuous map from $\R \to S^1$, so it may lifted to some $w_{x,\theta}: \R \to \R$. Let $\tilde{w}_{x,\theta}(t) :=  w_{x,\theta}(t) - w_{x,\theta}(0)$, so that $\tilde{w}$ does not depend on the lift $w$.

\begin{definition} [Pointwise Rotation Number]
	The average rotation number $\rho: X \to \R$ is defined by the limit:
	$$\rho(x) = \lim_{t \to \infty} \frac{\tilde{w}_{x,\theta}(t)}{t} ,$$
	whenever it exists, and is independent of the choice of $\theta$. 
\end{definition}

Indeed, for any $\theta, \theta' \in S^1$, we have $|\tilde{w}_{x,\theta}(t) - \tilde{w}_{x,\theta'}(t)| < 2\pi$ for any $t$ so the limit does not depend on choice of $\theta \in S^1$. 

Now define $\sigma: X \times \R \to \R$ and $\tau: X \times \R \to \R$ by:
$$\sigma^t(x) := \sup_{\theta \in S^1} \tilde{w}_{x,\theta}$$ $$\tau^t(x) := \inf_{\theta \in S^1} \tilde{w}_{x,\theta},$$
which, by continuity of $\cA$ are evidently continuous in $t$ and in $x$. Moreover, by the cocyle equation for $\cA$ it is clear that $\sigma$ is subadditive and $\tau$ is superadditive. 

By Kingman's subadditive ergodic theorem for flows, for any $\mu \in \cM_\Phi(X)$:
\begin{itemize}
	\item [(1)] The sequence $\frac{1}{t}\sigma^t$ converges $\mu$-a.e. to a $\Phi$ invariant map, which agrees with $\rho$. 
	\item [(2)] We may compute the integral of $\rho$ by:
	\begin{equation} \label{eq:subad} \rho_\mu: = \int \rho \, d\mu = \inf_{t > 0} \frac{1}{t} \int \sigma^t \,d\mu. \end{equation}
\end{itemize}

The discussion above then implies:

\begin{theorem} \label{rotationnumberiscontinuous}
	The map $\cM_\Phi(X) \to \R$ given by $\mu \mapsto \rho_\mu$ is continuous.
\end{theorem}
\begin{proof}
	Note that by compactness of $X$ and continuity of $\sigma^t:X \to \R$, the map  $\mu \to \int \sigma^t \,d\mu$ is continuous, and hence by:
	$$\int \rho \, d\mu = \inf_{t > 0} \frac{1}{t} \int \sigma^t \,d\mu 
	$$
	and the analogous equation for $\tau$, we obtain upper and lower semicontinuity of $\mu \mapsto \rho_\mu$.
\end{proof}

\begin{remark} \label{closedrotationdefinition}
	When $\mu$ is supported on a periodic orbit $\Or$, we will often write $\rho_\Or$ for $\rho_\mu$.
\end{remark}

Next we consider perturbations of cocycles over a fixed base flow. The space of cocycles $\cC_\Phi$ over the same $\Phi$ has a $C^0$-topology of uniform convergence defined by the property that $\cA_n \to \cA$ if for each $x \in X$ and $|t| <1 $ the maps $(A_n)_x^t \to A^t_x$ in $C^0(S^1, S^1)$ uniformly. 

Associated to the cocycles $\cA$ are rotation numbers $\rho_\mu(\cA)$ for invariant measures $\mu$ defined by Equation (\ref{eq:subad}). Then:

\begin{proposition} \label{contcocy}
	For a $\mu \in \cM_\Phi(X)$, the map $\cC_\Phi \to \R$ given by
	$$\cA \mapsto \rho_\mu(\cA)$$
	is continuous.
\end{proposition}
\begin{proof}
	The proof is nearly identical to that of Theorem \ref{rotationnumberiscontinuous}. Namely, one uses continuity of $\cA\mapsto \int \sigma^t(\cA) \, d\mu$ and the subadditive ergodic theorems.
\end{proof}

Now we specialize to the case where $X = \Or$ is a hyperbolic periodic orbit of a $C^1$ flow $\Phi_0$ on a Riemannian manifold $N$, which will be $N = SM$ with the Sasaki metric in the setting of this paper. We are interested in how $\rho_\Or$ varies as the flow $\Phi$ varies, for the derivative cocycle on certain circle bundles. 

By structural stability of the hyperbolic set $\Or$ there exists $\cU$ a $C^1$-neighborhood of $\Phi_0$ and a continuous $h: \cU \times \Or \to N$ such that the maps $h_{\Phi}(x) := h(\Phi, x)$ are $C^1$-diffeomorphisms onto their images, and $\Or_\Phi: = h_{\Phi}(\Or)$ is a closed orbit of $\Phi$. Moreover, since the maps $h_{\Phi}$ are $C^1$ there exists a continuous $\kappa: \cU  \times \Or \times \R\to \R$ such that $\kappa_{\Phi}(x, t) := \kappa(\Phi, x, t)$ is $C^1$ and the flow $\tilde{\Phi}$ (defined on $\Or_\Phi$) given by:	
$$\tilde{\Phi}^t(x) = {\Phi}^{\kappa_{\Phi}(h_\Phi(x),t)}(x),$$
is in fact conjugated to $\Phi_0$ by $h_{\Phi}$, i.e., $h_\Phi \circ \Phi_0 = \tilde{\Phi} \circ h_\Phi$. . 

For any bundle $E$, we write $\cP E$ for its projectivization. Let $F_0$ be a 2-dimensional trivial subbundle of $TN|_\Or$ which is part of a dominated splitting $E_0 \oplus_\leq F_0 \oplus _\leq G_0$ of $TN|_\Or$. The derivative cocycle $D\Phi$ on $\P F_0$ then is a cocyle on a trivial $S^1$ bundle, and it has a rotation number $\rho_\Or$ as before. 

Assuming $\cU$ is taken sufficiently small, by persistence of dominated splittings for each $\Phi \in \cU$ there is a splitting $TN|_{\Or_\Phi} = E_\Phi \oplus_\leq F_\Phi \oplus _\leq G_\Phi$ for $\Phi$ and the bundle $F_\Phi$ is trivial. Moreover, the splitting is also dominated for the flow $\tilde{\Phi}$, which is simply a time change of $\Phi$. Hence, $D\tilde{\Phi}$ and $D{\Phi}$ on $\P F_\Phi$ also have well defined rotation numbers $\rho_{\Or_{\tilde{\Phi}}}$, $\rho_{\Or_{\Phi}}$, which satisfy the relation:
$$\rho_{\Or_{\tilde{\Phi}}} \ell(\Or_\Phi) = \rho_{\Or_{\Phi}} \ell(\Or),$$
as they differ by a time change.

With all the objects defined, we now state the continuity with respect to the parameters: 

\begin{proposition} \label{rotcont2}
	The map $\cU \to \R$ given by 
	$$ \Phi \mapsto \rho_{\Or_\Phi},$$
	is continuous in some open $\cV \subseteq \cU$ containing $\Phi_0$.
\end{proposition}
\begin{proof}
First, we would like to consider all cocycles $D\tilde{\Phi}$ constructed on $F_\Phi$ as existing on the same bundle over the same base map. 

For $x \in \Or$ there exists a unique length-minimizing geodesic segment (from the Riemannian structure on $N$) from $x$ to $h_\Phi(x)$, as long as $h_\Phi$ is close to the identity, which may be ensured by passing to some $\cV \subseteq \cU$ further if needed. By parallel transport of the bundle $F_\Phi$ over $\Or_\Phi$ along such segments, one then obtains a $2$-dimensional trivial bundle $F_\Phi'$ over $\Or$. By shrinking $\cV$ further if needed, the bundle $F_\Phi'$ obtained is a given by a graph over $F_0$  with respect to the fixed Riemannian metric on $N$, and hence by orthogonal projection they may be identified. 

Since all maps above are continuous, the construction above describes a continuous map $Th: \cU \times F_0 \to TN$, so that $Th_\Phi(\cdot) := Th(\Phi, \cdot)$ are bundle isomorphisms $F_0 \to F_\Phi$ fibering over $h_\Phi$. Hence, conjugating by $Th_\Phi$ we may regard $D\tilde{\Phi}$ on $F_\Phi$ as a cocyle on $F_0$ over $\Phi_0$. By continuous dependence on $\Phi$, this defines a continuous map $\Phi \to D\tilde{\Phi}$, where $D\tilde{\Phi}$ are now regarded as elements of the space of cocycles over $\Phi_0$ on $F_0$ with the $C^0$-topology. 

Since all rotation numbers $\rho_{\Or_{\tilde{\Phi}}}$ defined previously are preserved by conjugation, it suffices to check continuity of the rotation numbers of the conjugated cocycles, which is given by Proposition \ref{contcocy}. Thus the map $\Phi \mapsto\rho_{\Or_{\tilde{\Phi}}}$ is continuous, and finally since 
$$\rho_{\Or_{\tilde{\Phi}}} \ell(\Or_\Phi) = \rho_{\Or_{\Phi}} \ell(\Or),$$
and the periods vary continuously, the map $\Phi \mapsto \rho_{\Or_\Phi}$ is continuous as well.
\end{proof}

\subsection{Geodesic Flows} Let $M$ be a smooth closed manifold. Since twe consider varying Riemannian metrics, it is useful to work on the sphere bundle over $M$ of oriented directions of the tangent space, which we denote by $SM$, rather than on the unit tangent bundle. When a metric $g$ is fixed, $T^1_g M$ is canonically diffeomorphic to $SM$, and one can pullback the Sasaki metric from $T^1_gM$  to $SM$. 

Recall that for $3 \leq k \leq \infty$ we denote by $\mathcal{G}^k$ the set of $C^{k}$-Riemannian metrics on $M$ with sectional curvatures $1 \leq -K < 4$. The geodesic flow on the unit tangent bundle of a negatively curved Riemannian manifold is an Anosov flow with the horospherical foliations corresponding to the stable and unstable foliations; moreover, under the pinching condition above it is a $\frac{1}{2}$-bunched (see Anosov flows section) Anosov flow \cite[Theorem 3.2.17]{k}. In particular, the bundles $E^{cu, cs}$ are $C^1$, and, since the flow is contact and the kernel of the $C^{k-1}$ contact form equals $E^u \oplus E^s$, in fact $E^u \oplus E^0 \oplus E^s$ is a (at least) $C^1$ Anosov spliting.

We describe now the perturbational results of \cite{klta} that will be used to perturb the derivative cocycle by perturbing the metric. For a fixed embedded compact interval or closed loop  $\gamma \subseteq SM$, the set of metrics for which $\gamma$ is an orbit segment of the geodesic flow is denoted by $\mathcal{G}_\gamma^k \subseteq \mathcal{G}^{k}$.  For a fixed $g_0 \in \G_\gamma^k$, pick local hypersurfaces $\Sigma_0$ and $\Sigma_1$ in $SM$ that are transverse to $\dot{\gamma}(t) \in TSM$ at $t = 0$ and $t = 1$, respectively. This allows us to define a Poincaré map $$P_{g_0} : \Sigma_0 \supseteq U \to \Sigma_1,$$where $U$ is a neighborhood of $\gamma(0)$, by mapping $\xi \in U$ to $\varphi_{g_0}^{t_1}(\xi)$, where $t_1$ is the smallest positive time such that $\varphi_{g_0}^{t_1}(\xi) \in \Sigma_1$. By the Implicit Function Theorem and the fact that $\varphi_{g_0}^{t}$ is $C^{k-1}$, the map $P$ is $C^{k-1}$.  

By projecting the tangent spaces of $\Sigma_{i= 0,1}$ to $E^u \oplus E^s$ one may give $\Sigma_{i =0,1}$ a symplectic structure which is preserved by the Poincaré map, since the symplectic form is invariant by the geodesic flow \cite{klta}. With $g_0$ fixed, we let $\G_{g_0, \gamma}^k \subseteq \G_\gamma^k$ be the set of metrics such that $\pi(\gamma(0)), \pi(\gamma(1)) \notin (g - g_0)$ ($\pi:SM \to M$ is the canonical projection map) that is, metrics unperturbed at the ends of the fixed geodesic segment $\gamma$ relative to $g_0$.

We will repeatedly use the main result on generic metrics established by Klingenberg and Takens in \cite{klta} to perturb the metric $g_0$:

\begin{theorem} \label{ktmain} \cite[Theorem 2]{klta} Suppose $g_0 \in \G^\infty_\gamma$, and let $Q$ be some open dense subset of the space of $(k-1)$-jets of symplectic maps $(\Sigma_0, \gamma(0)) \to (\Sigma_1, \gamma(1))$.
	
	Then there is arbitrarily $C^k$-close to $g_0$ a $g' \in \G_{ g_0, \gamma}^k$ such that $P_{g'} \in Q$, where $P_{g'}: (\Sigma_0, \gamma(0)) \to (\Sigma_1, \gamma(1))$ is the Poincaré map for the geodesic flow of $g'$.
\end{theorem}

\begin{remark}
	The technical assumption that $g_0$ is $C^\infty$ needed in \cite{klta} is virtually harmless, since by smooth approximation $\G^\infty \subseteq \G^k$ is dense for all $k$.
\end{remark}

We will need two additional facts about how these perturbations can be made, both of which follow directly from the proof of Theorem \ref{ktmain} in \cite{klta}:

\begin{proposition} \label{localpert}
	Let $h:= g'- g_0 \in S^2 T^*SM$, where $g'$ and $g_0$ given as in the statement of Theorem \ref{ktmain}. For any tubular neighborhood $V$ of $\gamma$, $h$ can be taken to satisfy:
	\begin{enumerate}
		\item [(1)] $\text{supp}(h) \subseteq V$;
		\item [(2)] For a system of coordinates $\{x_0, ..., x_{2n-2}\}$ on $V$ where $\partial_{x_0}$ is parallel to the geodesic flow,  the $k$-jets of $h_{00}$ (where $h = h_{ij} \, dx_i\,  dx_j$) vanish identically along $\{x_0 = 0\}$.
		
		In particular, this implies that the parametrization of  $\gamma$ by arc-length in $g_0$ is the same as that in $g'$, i.e., the geodesic flow for both metrics agree along $\gamma$.
	\end{enumerate}
\end{proposition}

Let $J^{k-1}_s$ denote the Lie group of $(k-1)$-jets of $C^{k-1}$ symplectic maps  $(\R^{2n}, 0) \to (\R^{2n}, 0)$ with the standard symplectic form $\sum_i dx^i \wedge dy^i$. If $\Or$ is a closed orbit, we may take $v := \gamma(0) = \gamma(1) \in \Or$ and fix $\Sigma := \Sigma_0 = \Sigma_1$, so by Darboux's theorem we may choose coordinates that identify the space of $(k-1)$-jets of  $C^{k-1}$ symplectic maps $(\Sigma, v) \to (\Sigma, v)$ with $J^{k-1}_s$.

\begin{corollary} \label{perturbclosed}
	If $\Or$ is a closed geodesic for $g_0 \in \G_\Or^\infty$ and $Q \subseteq J^{k-1}_s$ is an open dense invariant ($Q$ satisfies $\sigma Q \sigma^{-1} = Q$ for any $\sigma \in J^{k-1}_s$) set then there is arbitrarily $C^k$-close to $g_0$ a $g' \in \G_{\Or}^k$ such that for any $v \in \Or$ and any $\Sigma$ a transverse at $v$,  $P_{g'} \in Q$, where $P_{g'}  = P(v, \Sigma) $ is the Poincaré return map for the geodesic flow of $g'$.
\end{corollary}
\begin{proof}
	Choice of a different section $\Sigma$ or a different point $v$ of the orbit changes $P_{g'}$ by conjugation, so the property that $P_{g'} \in Q$ needs only be assured at one fixed point and one fixed section, which is done by Theorem \ref{ktmain}.
\end{proof}

\begin{remark} \label{1jets} Since the map $\pi^{k-1}: J^{k-1}_s \to J^1_{s} \cong \text{Sp}(2n)$ is a submersion, for $Q$ an open dense invariant subset of $\text{Sp}(2n)$, $(\pi^{k-1})^{-1}(Q)$ is an open dense invariant subset of $J^{k-1}_s$, so in the statement of Corollary \ref{perturbclosed} we may take an open dense invariant $Q \subseteq \text{Sp}(2n)$ instead, while the approximation is still in $\G^k$.
	
	In the context of Theorem \ref{ktmain}, the analogous observation holds; that is, one may take $Q$ to be an open dense subset of $1$-jets of symplectic maps $(\Sigma_0, \gamma(0)) \to (\Sigma_1, \gamma(1))$, and approximate in $\G^k$.
\end{remark}


\section{Pinching and Twisting for Flowss} \label{sec: pt}

In this section, we present the main technical results of the paper, namely, the construction of perturbations of Anosov flows leading to an appropriate pinching and twisting condition. For the sake of simplicity we specialize to the class of geodessic flows, but the main arguments here adapt to the proofs of the other theorems with adjustments which we describe in the last section. We define pinching and twisting for orbits of the geodesic flow in analogy with Definition \ref{deftyp}, and use the results on generic metrics to show that these are $C^1$-open and $C^k$-dense.

We fix the following useful notation. For a metric $g$ such that $\mathcal{O} \subseteq SM$ is a periodic orbit of its geodesic flow with period $\ell$, let $v \in \Or$ and let $\{\lambda_1, ..., \lambda_{2n}\}$ be the generalized eigenvalues of $D_v\varphi_g^\ell|_{E^u \oplus E^s}$, which do not depend on the choice of $v$,  sorted so that $|\lambda_i| \geq |\lambda_j|$ whenever $i < j$. We write:
$$ \exps^u(\Or, g): = (\lambda_1, ..., \lambda_n)  ,\, \hspace{.3cm} \exps^s(\Or, g): = (\lambda_{n+1}, ..., \lambda_{2n}) \in \C^n, 
$$
$$\, \exps (\Or, g): = (\lambda_1, ..., \lambda_{2n}) \in \C^{2n}.$$
The $i$-th coordinates of the vectors above are written as $\exps^{u, s, \cdot}_i(\Or, g)$ (where $\cdot$ means no superscript above).

The following continuity lemma about these $\exps$ is the bread and butter of all ``openness" arguments which follow:

\begin{lemma} \label{expscont} For a metric $g_0 \in \G^k$ there exists a neighborhood $\cU \subseteq \G^k$ of $g_0$ such that for any $g \in \cU$ any orbit $\Or$ of the geodesic flow of $g_0$ has a hyperbolic continuation $\Or_g$ for the geodesic flow of $g$, and the maps $\cU\to \C^n$ given by $$g \mapsto \exps^{u,s}(\Or_g, g)$$ are continuous with respect to the $C^2$-topology.
\end{lemma}
\begin{proof}
	Let $\Sigma$ be a smooth hypersurface parallel to $E^u \oplus E^s$ at $v$ so that $\Or \cap \Sigma =: \{v\}$. The return map for the geodesic flow $\varphi_{g_0}$ then defines a map $P_{g_0}: U \to \Sigma$, where $U\subseteq \Sigma$ is some neighborhood of $v$, for which $v$ is a hyperbolic fixed point. 
	
	For any $g$ sufficiently close to $g_0$, we also obtain a map $P_g: U \to \Sigma$ given by the return map of $\varphi_g$, and by the standard hyperbolic theory, a fixed point $v_g$ such that $g \mapsto v_g$ is continuous. The geodesic flow $\varphi_g$ varies in a $C^{k-1}$ fashion as $g$ varies in $\G^k$, and by the implicit functon theorem so does $P_g$. Then by fixing a coordinate system, since $k \geq 3$ the matrices $D_{v_g}P_g$ vary continuously, so their eigenvalues vary continuously as $g$ varies in $\G^k$. 
	
	Finally, the eigenvalues of the matrices $D_{v_g}\varphi_g^{\ell_g}|_{E^u \oplus E^s}$ and $D_{v_g}P_g$ agree, so we obtain the desired result.
\end{proof}

\subsection{Pinching}

Before moving to the definition of pinching, first we verify that generically there exists a periodic orbit with a dominated splitting of $E^u \oplus E^s$ into 1-dimensional subspaces and 2-dimensional subspaces corresponding to conjugate pairs of eigenvalues.

\begin{proposition} \label{gdopendense}
	Let 
	$$\mathcal{G}_d^k :=\{g \in \mathcal{G}^k: \exists\mathcal{O}:  |\lambda_i| \neq |\lambda_j|, \text{ unless } \lambda_i = \bar{\lambda}_j, \text{ where } (\lambda_1, ..., \lambda_n):=\exps^u(\Or, g)\}.$$
	
	The set $\G_d^k$ is $C^2$-open and $C^k$-dense in $\mathcal{G}^k$. 
\end{proposition}

\begin{proof}
	Openness follows directly from Lemma \ref{expscont}, since by continuity of $\exps^u$ the continuations of $\Or$ will satisfy the same condition defining $\G_d^k$.
	
	For density, we start by assuming that $g_0 \in \G^\infty_\mathcal{O}$, for some $\Or$, which is possible by density of $\G^\infty$ in $\G^k$.  It remains to check that the property defining $\G_d^k$ is indeed an open dense in $J^{k-1}_s$, so that we may apply Corollary \ref{perturbclosed} to finish the proof. By Remark \ref{1jets}, it suffices to check that having eigenvalues distinct with distinct norms, apart from complex conjugate pairs, is an openand  dense Sp$(2n)$.

Openness is clear, since the eigenvalues depend continuously on the matrix entries. For density, we note that  the condition of distinct eigenvalues is given by the complement of the equation $\Delta = 0$, where $\Delta$ is the discriminant of the characteristic polynomial,  which is a non-empty Zariski open set in Sp$(2n)$, and thus dense in the analytic topology. In particular, the set of diagonalizable matrices is dense. Since diagonalizable matrices are symplectically diagonalizable, by the lemma following this proof, by a small perturbation on the norm of the diagonal blocks we obtain density of eigenvalues of distinct norms.
\end{proof}

We prove the linear algebra lemma used above, which will also be useful in what follows:

\begin{lemma} \label{linalg}
	A matrix $A \in \text{Sp}(2n)$ with all eigenvalues distinct is symplectically diagonalizable in the sense that there exists $P \in \text{Sp}(2n)$ such that $P^{-1}AP$ is in real Jordan form (i.e., given by  diagonal blocks which are either trivial or $2\times 2$ conformal).
\end{lemma}
\begin{proof}
	Recall that eigenvalues of $A \in \text{Sp}(2n)$ appear in 4-tuples $$\{\lambda, \bar{\lambda}, \lambda^{-1}, \bar{\lambda}^{-1}\}$$ for $\lambda \notin \R$ and in pairs $\{\lambda, \lambda^{-1}\}$ for $\lambda \in \R$. For each $\lambda$ we let $E_\lambda = E_{\lambda^{-1}}$ be the 2-dimensional subspace spanned by the eigenspaces of $\lambda$ and $\lambda^{-1}$.
	
	Extend $\omega$ and $A$ to $\omega_\C$ and $A_\C$ in the complexification $\C^{2n} = \R^{2n} \otimes \C$. By definition $A_\C$ and $\omega_\C$ agree with $A$ and $\omega$ on $\R^{2n} \otimes 1$. 
	
	The identity for eigenvectors $v_\lambda$ and $v_\eta$:
	$$\omega_\C(v_\lambda, v_\eta) =  \omega_\C(A_\C v_\lambda, A_\C v_\eta) = \lambda\eta\, \omega_\C(v_\lambda, v_\eta),$$
	implies that, unless $\lambda\eta = 1$, we have $\omega_\C(v_\lambda, v_\eta) = 0$. Therefore $E_\lambda \otimes \C$ is symplectically orthogonal to $E_\eta \otimes \C$ for any $\lambda \neq \eta, \eta^{-1}$.
	
	In particular, this implies that the $E_\lambda \otimes 1$ are symplectic subspaces with respect to $\omega$ the real form, and symplectically orthogonal to each other. In each $E_\lambda$, $A$ can be put in Jordan real form with respect to a symplectic basis. By orthogonality we may construct a symplectic basis for $\R^{2n}$ by taking the union of symplectic bases for the $E_\lambda$. Then let $P$ be the matrix which sends the standard $\R^{2n}$ basis to the constructed symplectic basis.
\end{proof}

The next step is to construct a metric with a periodic orbit with simple real spectrum  with an arbitrarily small perturbation of the metric. Following \cite{bv}, this is accomplished by slightly perturbing a periodic orbit $\mathcal{O}$ rotating a complex eigenspace, and propagating the perturbation to a periodic orbit which shadows a homoclinic orbit of $\mathcal{O}$ that spends a long time near $\mathcal{O}$.

Recall the  following definitions: an \textit{$\eps$-pseudo-orbit} for a flow $\Phi$ on a space $X$ is a (possibly discontinuous) function $g : \R \to X$ such that:
	$$d(\gamma(t+\tau),\Phi^\tau(\gamma(t)))<\eps \, \text{ for } t \in \R \, \text{ and } \,|\tau| < 1. $$

 For $\gamma$ a $\varepsilon$-pseudo-orbit, we say $\gamma$ is said to be \textit{$\delta$-shadowed} if there exists a point $p \in X$ and a homeomorphism $\alpha: \R \to \R$ such that $\alpha (t) - t$ has Lipschitz constant $\delta$ and $d(\gamma(t), \Phi^{\alpha (t)}(p)) \leq \delta$ for all $t \in \R$. 
	

The classic closing lemma for Anosov flows we need is:

\begin{theorem} \cite{fh} (Anosov Closing Lemma) \label{closinglemma} If $\Lambda$ is a hyperbolic set for a flow $\Phi$ then there are a neighborhood $U$ of $\Lambda$ and numbers $\varepsilon_0, L > 0$ such that for $\varepsilon \leq \varepsilon_0$ any compact $\varepsilon$-pseudo-orbit in $U$ is $L\varepsilon$-shadowed by a unique compact orbit for $\Phi$.	
\end{theorem}

We use it to prove the main result of this section:

\begin{proposition} \label{pinch}
	Let
	$$\mathcal{G}_p^k := \{g \in \mathcal{G}^k: \exists \Or: \lambda_i \neq \lambda_j, \lambda_i \in \R, \textit{ where } (\lambda_1, ..., \lambda_n):=\exps^u(\Or, g). \}. $$
	In this situation, we say $\Or$ has the \textit{pinching property} for $g$. Then $\G_p^k$ is $C^2$-open and $C^k$ dense in $\G^k$.
\end{proposition}

\begin{proof}
	Fix a $C^2$-open set $\cU \subseteq \G^k$. First, since $\G_d^k$ is $C^2$-open and dense and $\G^\infty$ is $C^k$-dense in $\G^k$, we may fix some $g_0 \in \cU \cap \G_d^\infty$.  Let $\Or$ be as in Proposition \ref{gdopendense}. 
	
	Suppose that the vector $\exps^u(\Or, g)$ has $2c$ entries in $\C \setminus \R$, for some $c > 0$. It suffices to show that there exists a metric $g'$ in $\mathcal{U}$ which has a periodic orbit $\Or'$ such that $\exps^u(\Or', g')$ has $2(c-1)$ complex entries and all real entries distinct.
	
	Along $\mathcal{O}$ there is a dominated splitting $E^u = E^-_1 \oplus \dots \oplus E^-_k$ such that each $E_i$ is either 1 or 2-dimensional. Fix the smallest index $i \in \{1,..., k\}$ such that $E^\pm_{i}$ is 2-dimensional and let $P_{g_0}$ denote the Poincaré return map of the geodesic flow for a fixed section $\Sigma$ transverse to the flow small enough so that $\Or \cap \Sigma =:\{v\}$. By shrinking $\cU$ further if needed we may assume that $\cU \subseteq \G_d^k$, i.e., that the dominated splitting for $\varphi_{g_0}$ along $\Or$ persists for the continuation of $\Or$ for all $g \in U$; thus, by Lemma \ref{expscont} the map $g \mapsto \theta_g:= |\text{arg}(\lambda_g)|$ is well defined and continuous, where $\lambda_g$ is an eigenvalue of $D\varphi_g$ on $E^-_i$ on the continuation of $\Or$.
	
	
	\begin{lemma} There exists $g_1 \in \mathcal{U} \cap \G_\Or^k$ such that $\theta_{g_1} \neq \theta_{g_0}$.
	\end{lemma}
	
	\begin{proof}
		The derivative of the Poincaré map is conjugate to $D\varphi_g|_{E^u \oplus E^s}$ over the closed orbit $\Or$, so $\theta_{g_0}$ agrees with the argument of the eigenvalue of $DP_{g_0}$ along the 2-dimensional Jordan block $F^- \subseteq T_v\Sigma$ mapped to $E_{i}^-$ under the conjugation aforementioned.  Moreover, let $F^+$ be the Jordan block corresponding to $E^+_i$ in the same manner.
			
		Identifying the space of symplectic maps $T_v\Sigma \to T_v\Sigma$ with Sp$(2n)$ there exists some neighborhood $\cV \subseteq \text{Sp}(2n)$ of the original map $DP_{g_0}$, such that for $A \in \cV$ the Jordan block $F$ has a continuation for $A$, and we call the norm of the argument of the eigenvalue of $A$ along this continuation $\theta_A$. Let $\cW \subseteq \cV$ be the set of matrices $A$ such that $\theta_A \neq \theta_{g_0}$. If $\cW$ is open and dense in $\cV$ then by Remark \ref{1jets} we may apply Corollary \ref{perturbclosed} to $\cW \cup ((\text{Sp}(2n)\setminus \text{Cl}(\cV))$, which will be open and dense in Sp$(2n)$ to find that the set of metrics which has $\theta_{g_1} \neq \theta_{g_0}$ is dense (and open) in $\cU$. 
		
		It remains to check that $\cW$ is open and dense in $\cV$. Openness is clear by continuous dependence of eigenvalues on matrix entries. For density, let $R_\theta$ be given by rotation of any angle of $\theta > 0$ on the subspaces $F^-, F^+$ and the identity on the other subspaces, satisfies $R_\theta \Omega R_\theta^{T} = \Omega$, where $\Omega$ is the standard symplectic form.  Then $R_\theta DP_{g_0}$ has $\theta_{R_\theta DP_{g_0}} \neq \theta_{g_0}$; since $\theta > 0$ can be made arbitrarily small, this finishes the proof.
	\end{proof}
	
	Let $g_1$ be given as in the lemma above, and for $0 \leq s \leq 1$ we let $g_s = sg_1 + (1-s)g_0$, which, if $g_1$ is taken sufficiently close to $g_0$, also satisfies $\{g_s\} \subseteq \mathcal{U} \cap \G_\Or^k$. Clearly, the map $[0,1] \to \G^k$ given by $s \mapsto g_s$ is continuous. Also note that, by Proposition \ref{localpert} (2), $\Or$ is not only a closed orbit of $\varphi_{g_s}$ for all $s \in [0,1]$, but it in fact has the same arc-length parametrization with respect to all $g_s$.
	
	For the geodesic flow of $g_0$, fix $w$ a transverse homoclinic point of $v$, i.e., $w \in W^{u}(v) \cap W^{cs}(v)$. Fix some $\eps > 0$ so that the geodesic flow has local product structure at scale $2\eps$. Then there exists $t_1, t_2 > 0$ such that $\varphi^{-t_2}_{g_0}(w) \in W^u_{\eps}(v)$, $\varphi^{t_1}_{g_0}(w) \in W^s_{\eps}(v)$ and also a $C> 0$ such that for all $t > 0$: $$d(\varphi^{-(t_2+t)}_{g_0}(w), \varphi_{g_0}^{-t}(v)) < C\eps e^{-t},$$ $$d(\varphi^{t_1+t}_{g_0}(w), \varphi_{g_0}^t(v)) < C\eps e^{-t}.$$
	
	Hence for $n \in \N$ the $\gamma_n: \R \to SM$ given by $$\gamma_n(t) = \varphi^{\tilde{t}-(t_2+n\ell)}_{g_0}(w), \text{ where }\tilde{t} = t \text{ mod } (t_2 + t_1 + 2n \ell)$$ are $\eps_n$-pseudo-orbits where $\eps_n < 2C\eps e^{-n\ell}$, by the fact that the minimal expansion of the geodesic flow is $\tau = 1$ by the assumption on curvature. 
	
	For $n$ sufficiently large, there exist unique periodic $w_n$'s which $L\eps_n$-shadow $\gamma_n$. Let $w_{n,s}$ be continuations of $w_n$ for $0 \leq s \leq 1$ (where $w_{n,0} = w_n$, by definition). Let $w_s$ be the hyperbolic continuations of $w$. By uniqueness of shadowing, note that the $w_{n,s}$ can also be constructed by shadowing segments of the orbit of $w_s$. The following proposition shows we can extend the dominated splitting of $\Or$ to the new orbits we defined:
	
	\begin{lemma} \label{domsplit}
	   There exists $N$ large so that for each $0 < s < 1$ the compact invariant set 
		$$K_{N,s} = \bigcup_{n \geq N} \mathcal{O}(w_{n,s}) \cup \Or(w_s) \cup \Or,$$
		for the geodesic flow $\varphi_{g_s}$ of $g_s$ admits a  dominated splitting for the bundle $E^u = E^{-}_{s,1} \oplus \dots \oplus E^{-}_{s,k}$ over $K_{m,s}$ coinciding with the dominated splitting of $E^u$ over $\mathcal{O}$, and similarly for $E^s = E^{+}_{s,1} \oplus \dots \oplus E^{+}_{s,k}$.
	\end{lemma}
	
	\begin{proof}[Proof sketch. See \cite{bv}, Lemma 9.2]
		We sketch the proof for $s = 0$ which is almost identical to the result cited. Then since dominated splittings over compact invariant sets persists under $C^1$-small perturbations by an invariant cone argument, this shows the result for all $s \in [0,1]$. Consider the case of $E^u$.
		
		Since $w\in  W^{u}(v) \cap W^{cs}(v)$, one can extend the dominated splitting of $\Or$ to $\mathcal{O}(w)$ as follows. Consider the bundles over $\mathcal{O}$ given by $F^i =  E^{-}_{1} \oplus \dots \oplus E^{-}_{j+1}$, and $G^i = E^{-}_{j} \oplus \dots \oplus E^{-}_{k}$ for $i,j= 1, ..., k-1$. Then we define 		
		$$E^j(w) : = \phi^{cs}_{v, w} F^j(v) \cap \phi^{cu}_{v,w}G^j(v)$$
		and extend the $E^j$ bundles to $\Or(w)$ by the derivative of the flow. Proof of continuity and domination of this splitting follows closely that in \cite{bv}.
		
	    For $N$ sufficiently large we observe that $K_{N,0}$ is contained in an arbitrarily small neighborhood of $\Or \cup \Or(w)$, so the dominated splitting extends by continuity.
	\end{proof} 
	
	For each $n$, we let $\theta_{n}: [0,1] \to S^1 = \R/2\pi\Z$ be defined by setting $\theta_n(s)$ to be the argument of the eigenvalue of $D\varphi_{g_s}$ along $E_{s,i}^{-}$ on the closed orbit $w_{n,s}$. By Lemma \ref{expscont}, the $\theta_n$ are continuous so for each $n$ they may be lifted to some $\tilde{\theta}_n: [0,1] \to \R$.  
	
	The main result about these rotation numbers, whose proof is postponed to the next section due to its length, is:
	
	\begin{lemma} \label{rotationlemma}
		There exists $n \in \N$ so that $|\tilde{\theta}_n(1) - \tilde{\theta}_n (0)| > 2\pi $. 
	\end{lemma}
	
	By continuity one then finds $n,s$ such that $\tilde{\theta}_{n,s}$ is an integer multiple of $2\pi$, i.e., such that the eigenvalues in $\exps^u(\Or(w_{n,s}), g_s)$ corresponding to the subspace $E^-_{i,s}$ are real. By another perturbation using Corollary \ref{perturbclosed}, there exists a metric such that these eigenvalues become distinct. Then by induction on the other eigenspaces with complex eigenvalues, all eigenvalues are real and distinct.

	To finish the proof, openness follows again by Lemma \ref{expscont}, since the requirements on the products of the eigenvalues is an open condition. 
	\end{proof}

\subsection{Proof of Lemma \ref{rotationlemma}} \label{ssec:lemma} We apply the notions introduced in Section \ref{ssec:rotation} to give a proof of Lemma \ref{rotationlemma}.

\begin{proof} [Proof of Lemma \ref{rotationlemma}]

Fix $N$ large enough so that $K_{N,s}$ satisfies the conclusion of Lemma \ref{domsplit}.  We begin with:

\begin{proposition} \label{trivialization}
	There exists $N' > N$, which we denote by $N$ after this proposition, such that the bundles with total spaces $E_s$ defined by the fibers $E_s(x) := E_{s,i}^-(x)$ over $x \in K_{N',s}$ are continuously trivializable, for each $s \in [0,1]$. 
\end{proposition}

\begin{proof}

First, note that it suffices to prove that $E_s$ is trivializable for $s = 0$, since the bundles $E_s$ vary continuously in the ambient space $TSM$ as $s$ varies.  

We will construct a non-vanishing section of the frame bundle $F$ associated to $E_0$ over some $K_{N',0}$ for $N'$ large, which is equivalent to a continuous choice of basis for $E_0$, proving triviality of the bundle. 

\begin{figure}[ht]
	\begin{center}
		\includegraphics[scale=.18 ]{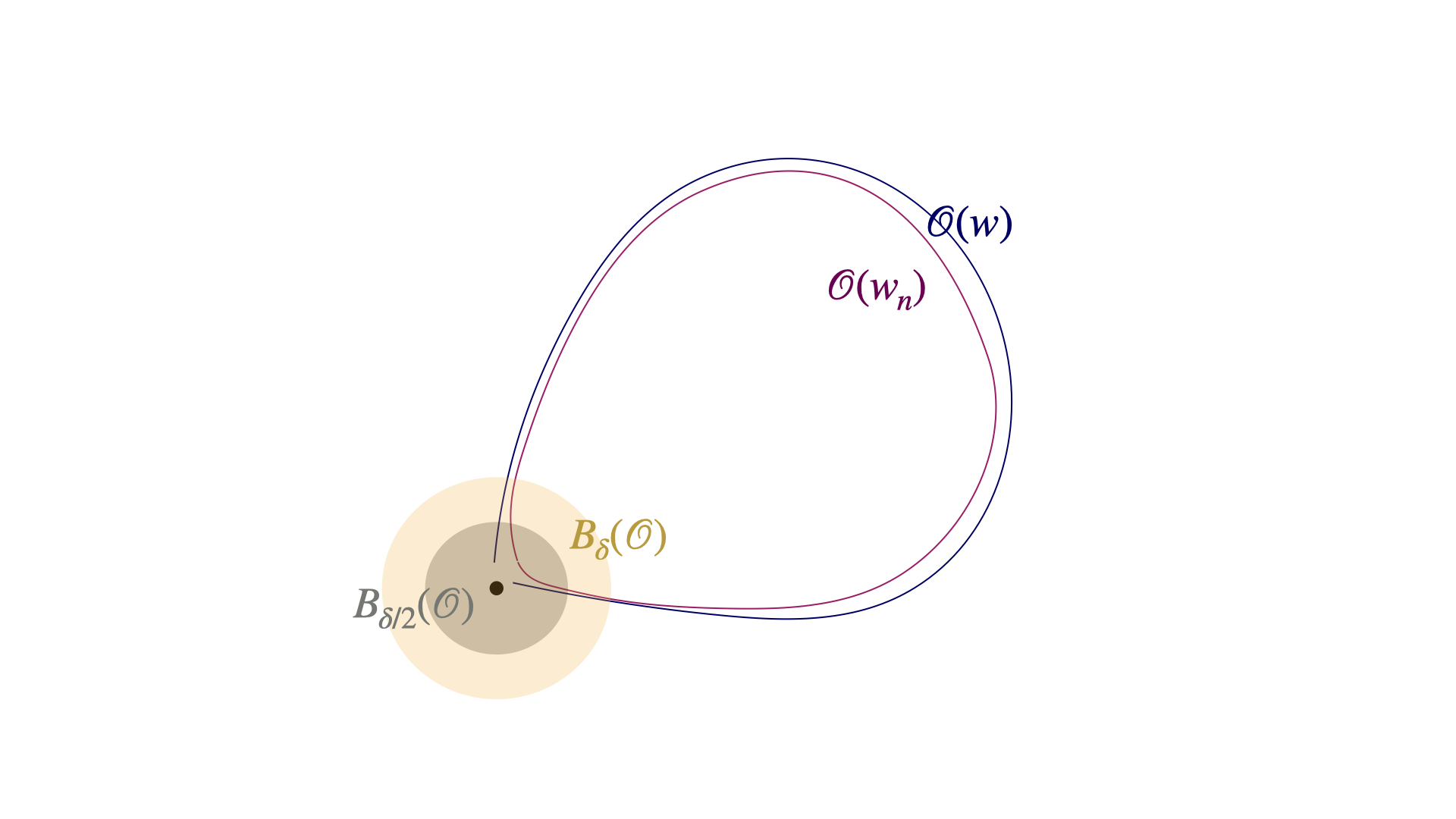}
	\end{center}
	\caption{Proof of Proposition \ref{trivialization}. The closed orbit $\Or$ is schematically represented by the black dot.}
\end{figure}

For $\delta > 0$ let $B_\delta(\Or)$ a $\delta$-tubular neighborhood of $\Or$. If  $\delta$ is sufficiently small relative to the scale of local product structure of the Anosov flow, for all $n \geq N'$, and $N'$ sufficiently large, $\Or(w_n)_\delta := B_\delta(\Or) \cap \Or(w_{n})$ consists of a connected segment of the embedded circle $\Or(w_{n})$ and moreover, $\Or(w)_\delta := B_\delta(\Or) \cap \Or(w)$ consists of the complement of a connected closed interval in $\Or(w)$, i.e.,  two immersed connected components (see Figure 1).

Note that we may assume that the return map of $D\varphi_{g_0}^{\ell(w_n)}$ is orientation preserving on $E_0$ over any periodic orbit $\Or(w_n)$, since otherwise it would have real eigenvalues (any $A \in \text{GL}(2, \R)$ with negative determinant has real eigenvalues) and we would obtain a proof of Lemma \ref{rotationlemma}. Hence, the bundle $E_0$ is trivializable over any $\Or(w_{n})$. It is also clearly so over $\Or(w)$, since it is an immersed real line, and we may assume it is too for $\Or$, since otherwise, again, we would have real eigenvalues. 

By shrinking $\delta$ further if necessary, there exists a well-defined closest point projection $p: B_\delta(\Or) \to \Or$ which is a surjective submersion. Fix a trivialization of $E_0$ over $\Or$, i.e., a non-vanishing section $S: \Or \to F$, which is possible by the previous paragraph. 

For $x \in B_\delta(\Or) \cap K_{N', 0} =: K_\delta$, again shrinking $\delta$ further if necessary, there exists a unique length-minimizing geodesic segment between $x$ and $p(x)$, and by parallel transporting $E_0(p(x))$ along such segments and then projecting orthogonally onto $E_0(x)$ one obtains a continuous bundle map $E_0|_{K_\delta}  \to E_0|_\Or$ which is an isomorphism on fibers. This map induces a map $F|_{K_\delta}  \to F_\Or$ and so by pulling back the non-vanishing section $S: \Or \to F$ we obtain a non-vanishing section, which we now denote by $S: K_\delta \to F$ since its restriction to $\Or$ agrees with the previous $S$, of $F$ over $K_\delta$.

Recall that $\Or(w)_\delta$ consists of two connected immersed components homeomorphic to $\R$. Since $\Or(w)$ is contractible, it is possible to define a determinant on $F|_{\Or(w)}$; up to scalar it is unique, and hence there is a well defined continuous sign function on each fiber. Then we claim that $S|_{\Or(w)_\delta}$ has the same determinant sign on both components, so that it may be extended to a continuous section $\Or(w) \to F|_{\Or(w)}$. Suppose not for a contradiction.	

 Define the line bundle $L:= \bigwedge^2 E^0$ over $K$, which restricted to individual orbits is trivial since $E_0$ is. At each point $x \in K$ there is a natural map $F(x) \to L(x)$ given by $(e_1, e_2) \mapsto e_1 \wedge e_2$, which extends to a continuous global map $W: F \to L$. Considering the image of $S|_{\Or(w)_\delta}$ under $W$, we obtain a section $\Or(w)_\delta \to L$, which has opposite signs in the two connected components. Let $B: \Or(w) \to L$ be any extension of this section to all of $\Or(w)$; by the previous remark, $B$ must have an odd number of zeros.
 
 By continuity $$\Or(w_n) \setminus B_{\delta/2}(\Or) \to \Or(w) \setminus B_{\delta/2}(\Or)$$ as $n \to \infty$, so by continuity of the bundle for $n$ sufficiently large we can parallel transport the section $B$ on $\Or(w) \setminus B_{\delta/2}(\Or)$ to $\Or(w_n) \setminus B_{\delta/2}(\Or)$ to obtain a section $B_n$ on $\Or(w_n) \setminus B_{\delta/2}(\Or)$  which has the same number of zeros as $B$ on $\Or(w) \setminus B_{\delta/2}(\Or)$, i.e., oddly many. 
 
 On the other hand as $n \to \infty$, $$B_n|_{\Or(w_n)_\delta \setminus B_{\delta/2}(\Or)} \to (W\circ S)|_{\Or(w_n)_\delta \setminus B_{\delta/2}(\Or)},$$ and hence for $n$ large enough $B_n$ has constant sign on  $\Or(w_n)_\delta \setminus B_{\delta/2}(\Or)$. Thus $B_n$ extends to $\Or(w_n)_\delta$ without any zeros. Hence we obtain a global section $B_n$ on $\Or(w_n)$ with an odd number of zeros, contradicting the triviality of $L$ over $\Or(w_n)$. 
 
 Hence we may extend $S|_{\Or(w)_\delta}$ continuously to all of $\Or(w)$. Since in $K_\delta$ the section $S$ is continuous, and again $\Or(w_n) \setminus B_{\delta/2}(\Or) \to \Or(w) \setminus B_{\delta/2}(\Or)$ as $n \to \infty$, we can then continuously extend $S|_{\Or(w)\setminus B_{\delta/2}(\Or)}$ to $\Or(w_n) \setminus B_{\delta/2}(\Or)$ while agreeing with $S$ in $K_\delta$. Since $S|_{\Or(w)}$ is non-vanishing, the section obtained in this way is also globally  non-vanishing.

\end{proof} 


	

The projectivization $\P E_{s}$ of the bundle $E_{s}$ then defines a trivial circle bundle over $K_{N,s}$, and we fix a trivializing bundle isomorphism $\phi_{s}: \P E_s \to K_{N,s} \times S^1$. By conjugating with $\phi_{s}$, the derivative of the geodesic flow then defines a continuous cocycle $\cA_{s}$ on $K_{N,s} \times S^1$ over the geodesic flow, so we may apply the results of Section \ref{ssec:rotation} for $\cA_{s}$. 

Then the the rotation numbers have the following characterization over periodic orbits:
\begin{lemma} \label{rotform}
	For a closed orbit $\Or(u)$ of a point $u \in K_{N,s}$, the argument $\theta(u)$ of the eigenvalue  of the return map of the geodesic flow on $E^-_{s,i}$ satisfies:
	$$\theta(u) = \ell(u) \cdot \rho_{\Or(u)} \, (\text{mod} \, \, 2\pi)$$
	where $\ell(u)$ is the period of $u$, and $\rho_{\Or(u)}$ is as in Remark \ref{closedrotationdefinition} for the cocycle $\cA_{s}$ defined above.
\end{lemma}
\begin{proof}
	On one hand, it follows from the definition of $\rho$ that $ \ell(u) \cdot \rho_{\Or(u)}$ agrees mod $2\pi$ with the Poincaré rotation number for the map $$(A_s)_u^{\ell(u)}: S^1 \to S^1.$$
	
	On the other, the projectivization of the derivative of the flow also defines on the fiber a homeomorphism $S^1 \to S^1$ with Poincaré rotation number equal to the argument of the eigenvalue of the derivative. 
	
	Since the two above differ by a conjugation given by $\pi_2\circ \phi_{s}(u, \cdot): S^1 \to S^1$, where $\pi_2: K_{N,s} \times S^1$ is the natural projection, by invariance we obtain the result.
\end{proof}

Applying Lemma \ref{rotform} to the $\theta_n(s)$, we obtain for $0 \leq s \leq 1$: 

$$\theta_n(s) = \ell(w_{n,s})\rho_{\Or(w_{n,s})}\,  (\text{mod}\, 2\pi),$$ 

By continuity of the functions $\theta_n$ we may lift them to $\tilde{\theta}_n: [0,1] \to \R$ satisfying $\tilde{\theta}_n(0) = \ell(w_{n,0})\rho_{\Or(w_{n,0})}.$
By the continuity of $\rho_{\Or(w_{n,s})}$ in $s$, given by Proposition \ref{rotcont2}, our choice of lift then implies:
$$\tilde{\theta}_n(s) = \ell(w_{n,s})\rho_{\Or(w_{n,s})}, \text { for } 0 \leq s \leq 1. $$

Let $\theta(s)$ be the argument of the eigenvalue of the $D\varphi_{g_s}$ on $E^-_{i,s}$ on the periodic orbit $\Or$ (recall $\Or$ is a closed geodesic for all $g_s$ with $\ell(\Or)$ fixed), and repeat the constructions above to obtain $\tilde{\theta}(s)$ as well satisfying
\begin{equation} \label{eq: thetaellrho}
\tilde{\theta}(s) = \ell(\Or) \rho_\Or(s),
\end{equation}
where $\rho_{\Or}(s)$ is $\rho_\Or$ of the geodesic flow of $g_s$.

Since $\mu_{\Or(w_{n,s})}  \to \mu_{\Or}$ (where $\mu_\Or$ is the invariant probability measure supported on the closed orbit $\Or$) we have $\rho_{\Or(w_{n,s})} \to \rho_{\Or}(s)$ as $ n \to \infty$ by Theorem \ref{rotationnumberiscontinuous}.  By hypothesis $\theta(1) \neq \theta(0)$, and since $\ell(\Or)$ is constant as $s$ varies, Equation (\ref{eq: thetaellrho}) gives that $\rho_\Or(1) - \rho_\Or(0) \neq 0$. Hence for $n$ large enough there exists some $\delta > 0$ such that $|\rho_{\Or(w_{n,1})}- \rho_{\Or(w_{n,0})}|  \geq \delta$.

Finally, let $\delta_n = |\ell(w_{n,1}) - \ell(w_{n,0})|$. Again, we defer the proof of the following final proposition we need:

\begin{proposition} \label{boundlength}
	There exists $M_2 > 0$ such that $\delta_n < M_2$ for all $n \in \N$.
\end{proposition}

With Lemma \ref{boundlength}, we complete the proof of Lemma \ref{rotationlemma}:
$$
\begin{aligned}
|\tilde{\theta}_n(1) -  \tilde{\theta}_n(0)| &= |\ell(w_{n,1})\tilde{\rho}_{\Or(w_{n,1})} - \ell(w_{n,0})\tilde{\rho}_{\Or(w_{n,0})}|  \\
&\geq |\ell(w_{n,1})(\tilde{\rho}_{\Or(w_{n,1})} - \tilde{\rho}_{\Or(w_{n,0})}) | \\
 &\,\,\,- |(\ell(w_{n,1})-\ell(w_{n,0}))\tilde{\rho}_{\Or(w_{n,0})} |\\
&\geq \delta \ell(w_{n,1}) - \delta_n |\tilde{\rho}_{\Or(w_{n,0})} | \\ 
&> \delta \ell(w_{n,1}) - M _1M_2> 2\pi
\end{aligned}
$$ 
for all $n$ sufficiently large, since $\ell(w_{n,1}) \to \infty$.
\end{proof}

At last, we prove Proposition \ref{boundlength}.

\begin{proof}[Proof of Proposition \ref{boundlength}] 

To bound the variations $\delta_n$, we use exponential shadowing and Hölder continuity of the geodesic stretch, defined below. Since the geodesic flow is unperturbed on $\Or$ and the orbits $\Or(w_{n,s})$ approximate $\Or$, the two mentioned properties give us the bound on $\delta_n$.

 Recall that the $w_n$ are constructed by shadowing $\gamma_n: \R \to SM$ given by $$\gamma_n(t) = \varphi^{\tilde{t}-(t_2+n\ell)}_{g_0}(w), \text{ where }\tilde{t} = t \text{ mod } (t_2 + t_1 + 2n \ell), $$
 which is a $\eps_n$-pseudo-orbit, where $t_1$ (resp. $t_2$) is such that $\varphi^{t_1}_{g_0}(w)$ (resp. $\phi^{-t_2}_{g_0}(w)$) is in $W^s_{\eps}(v)$ (resp. $W^u_{\eps}(v)$) and $\eps_n < 2C\eps e^{-n\ell}$. 
 
 The following well-known theorem is an adaptation for flows of the usual ``exponential" shadowing theorem, which uses the Bowen bracket in its proof. The statement gives a sharper estimate on how well shadowing orbits approximate pseudo-orbits:
 
 \begin{theorem} \cite[Theorem 6.2.4]{fh} For a hyperbolic set $\Lambda$ of a flow $\Phi$ on a closed manifold $\exists c,\eta>0$ such that $\forall \eps >0, \, \exists \delta>0$ so that: if $x,y\in \Lambda$, $s:\R\to\R$ continuous, $s(0) = 0$ and $d(\Phi^t(x),\Phi^{s(t)}(y)) < \delta $ for all $|t| \leq T$, then
 	\begin{enumerate}
 		\item [(1)] $|t - s(t)| < 3 \eps$ for all $|t| \leq T$, 
 		\item [(2)] there exists $t(x,y)$ with $|t(x,y)| < \eps$ so that the $\eps$-stable manifold $\Phi^{t(x,y)}(x)$ intersects uniquely the $\eps$-unstable manifold of $y$ and:
 		$$ d(\Phi^t(y), \Phi^t(\Phi^{t(x,y)} (x))) < c e^{\eta(T-|t|)} \text{ for } |t| < T.$$
 	\end{enumerate}
 \end{theorem}

In the context of the current proof, we apply the above theorem as follows. 

Let $T_n = \ell(w_n)$, $x = \varphi^{T_n/2}_{g_0} (w_n)$ and $y = \varphi^{\tau_n}_{g_0} (w)$ where $\tau_n  := T_n/2 - (t_2 + n\ell)$. For $n$ sufficiently large, $d(\varphi^t_{g_0}(x),\varphi^{s(t)}_{g_0}(y)) < \delta$ is satisfied, by the statement of shadowing, for $|t| < T_n/2$ and $\delta$ given by the theorem for the $\eps > 0$ fixed before.  Then the theorem gives a $t_n \in \R$ such that:
$$ d(\varphi^{t_n+t}_{g_0}(w_n),  \varphi^{\tau_n + t}_{g_0} (w) ) <c e^{\eta(T_n/2-|t|)}, \text{ for } |t| < T_n/2.$$

Now we turn to computing the period of $w_{n,1}$ using the facts established above. By structural stability, there exists $h: SM \to SM$ which conjugates the orbits of $\varphi_{g_0}$ to those of $\varphi_{g_1}$. This conjugacy can be taken to be Hölder continuous and $C^1$ along the flow direction. Thus, there exists some $a: SM \to \R$ which is Hölder continuous with some exponent $1 \geq \beta > 0$, such that for $u \in SM$:
$$dh (u) X_{g_0}(u) = a(u) X_g(h(u)),$$
where $X_g$ (resp. $X_{g_0}$) is the vector field generating the geodesic flow for $g$ (resp. $g_0$). The function $a$ is referred to as the \textit{geodesic stretch}, and the proof of the facts above can be found, for instance, in \cite[p. 12-13]{stretch}

The period of $w_{n,1}$ is given by the formula:
$$\ell(w_{n,1}) =  \int_0^{T_n} a(\varphi^t_{g_0}(w_n)) \, dt$$

By Proposition \ref{localpert} (2), since $\Or$ is a closed geodesic, with same arclength parametrization for $g_0$ and $g_1$, it is clear that $a|_\Or \equiv 1$.  Therefore, we may compute the difference $\delta_n = |\ell(w_{n,1})  - \ell(w_{n,0})|$ as follows:
$$ \begin{aligned}
|\ell(w_{n,1})  - \ell(w_{n,0})| &\leq \int_{-T_n/2}^{T_n/2} |a(\varphi^t_{g_0}(w_n)) - 1| \, dt \\
&\leq M \int_{-T_n/2}^{T_n/2} d(\varphi^{t_n+t}_{g_0}(w_n),  \Or)^\beta \, dt,
\end{aligned}$$
since $a$ is $\beta$-Hölder continuous and the distance between a point and a compact set is well defined. To estimate the distance, note:
 $$\begin{aligned} d(\varphi^{t_n+t}_{g_0}(w_n),  \Or) &\leq  d(\varphi^{t_n+t}_{g_0}(w_n),  \varphi^{\tau_n+ t}_{g_0} (w) )  + d(\varphi^{\tau_n+ t}_{g_0} (w) , \Or) \\
 &\leq  c (e^{\eta(T_n/2-|t|)} + e^{-|t|}), \text{ for } |t| < T_n/2,
 \end{aligned},$$
 since $w$ is a homoclinic point of $\Or$ so $d(\varphi^{\tau_n+ t}_{g_0} (w) , \Or) \leq ce^{-|t|}$ for some $c > 0$ which we assume, by taking the max if necessary, is the same as the previous $c$. Substituting this inequality into the previous integral, we obtain:
 $$|\ell(w_{n,1})  - \ell(w_{n,0})| \leq M\int_{-T_n/2}^{T_n/2}  (e^{\eta(T_n/2-|t|)} + e^{-|t|})^\beta \, dt < M_2 < \infty, $$
 for $M_2$ independent of $n$, as an easy calculus exercise shows.
\end{proof}

\subsection{Twisting}Following the previous section, we fix a metric $g_0 \in \G_p^k$. Let $\Or$ be the orbit with the pinching property, $v \in \Or$ and $l$ the period of $\Or$.We fix an arbitrary $w \in W^{cs}_{g_0}(v) \cap W^{cu}_{g_0}(v)$ a transverse homoclinic point of the orbit of $v$, and consider the holonomy maps $$\psi^{g_0}_{v,w} = h^{cs}_{w,v} \circ h^{cu}_{v,w}$$ given from Theorem \ref{holsflow}, for the unstable bundle $E^u$. Recall that $\exps^u(\Or, g)$ consists of distinct real numbers, so let $\{e_i\}$ be an (non-generalized, real) eigenbasis for $E^u$. For all $1 \leq j \leq k$ the alternating powers $\Lambda^j E^u(v)$ have a basis obtained as exterior products of the $e_i$. We write $e_I^k := e_{i_1} \wedge \dots \wedge e_{i_k}$, where $I = \{i_1, ..., i_k\}$.

\begin{proposition} \label{twist}
	
		For $g_0 \in \G_p^k$ as above we say $g_0$ has the \textit{twisting property} for $w \in SM$ with respect to $v$, and we write $g_0 \in \G_{p,t}^k$, if
	$$\forall e^k_I, e^l_{I'}, k + l  = n: (\wedge^k \psi^{g_0}_{v,w}) (e^k_I) \wedge e^l_{I'} \neq 0, $$
	which is to say that the image of any direct sums of eigenspaces intersects any direct sum of eigenspaces of complementary dimension only at the origin.
	
	The set $\G_{p,t}^k$ is $C^2$-open and $C^k$-dense in $\G^k$.
\end{proposition}
\begin{proof}
	
	Again, by density of $\G^\infty \subseteq \G^k$ and openness of $\G_p^k$ we may assume that $g_0 \in \G^\infty$ so we can apply Theorem \ref{ktmain}.	For some small $\eps > 0$, consider the geodesic segment $\gamma = \varphi^{g_0}_{[0, \eps]}(w)$. Note that since $\Or(w)$ accumulates as $|t| \to \infty$ on the compact set $\Or$, if we take $\eps > 0$ small enough we may take $\pi(\gamma)$ to be disjoint from $\pi(\Or(w) \setminus \gamma) \cup \pi(\Or)$, where $\pi: SM \to M$ is the projection map. 
	
	Then we apply Theorem \ref{ktmain} to $\gamma' \subseteq \gamma$, where $\gamma' =  \varphi^{g_0}_{[\delta, \eps-\delta]}(w)$ for $\delta > 0$ small, to perturb $D_w \varphi^\eps_{g_0}$ by perturbing the metric only on a tubular neighborhood $V_{\gamma'}$ of $\gamma'$ small enough (possible by Proposition \ref{localpert} (1)) so that $$V_{\gamma'} \cap \text{Cl}(\pi(\Or) \cup \pi(\Or(w) \setminus \gamma)) = \varnothing.$$ 
	where Cl denotes closure. 
	
	By equivariance of holonomies the map $\psi^{g_0}_{v,w}$ can be rewritten as:
	$$\psi^{g_0}_{v,w} =   h^{cs}_{\varphi^\eps_{g_0}(w),v} \circ D_w{\varphi}_{g_0}^{\eps}|_{E_u} \circ  h^{cu}_{v,w}.  $$
	
	Then observe that perturbations to the metric of the form described in the previous paragraph affect only the $D_w{\varphi}_{g_0}^{\eps}|_{E_u}$ term in the composition above. Indeed, we recall that $h^{cu}_{w,v}$ depends only on the values of the cocycle  on a neighborhood of the $(-\infty, 0]$ part of the orbit $\varphi^t_{g_0}(w)$, and $h^{cs}_{\varphi^\eps_{g_0}(w),v}$ on a neighborhood of the $[\varepsilon, \infty)$ part of the orbit $\varphi^t_{g_0}(w)$ and on the cocyle along $\Or$. By construction of $V_{\gamma'}$, the cocyle is not perturbed in any of these sets.
	
	It remains to check that for an open and dense set of $1$-jets of symplectic maps $P$ from a small transversal to the flow at $w$ to a small transversal section to the flow at $\varphi_{g_0}^\eps(w)$ the map $\psi^{g_0}_{v,w}$ has the twisting property (we assume both transversals to be tangent to $E^u$ at $w$ and at  $\varphi_{g_0}^\eps(w)$, respectively), if we replace $D_w{\varphi}_{g_0}^{\eps}|_{E_u}$ by $DP|_{E^u}$. This implies by Theorem \ref{ktmain} that we can construct such a small perturbation in the space of metrics, completing the proof.
	
	Since both holonomy maps in the composition defining $\psi^{g_0}_{v,w}$ as above are symplectic isomorphisms, an open and dense subset of $\text{Sp}(E^u(v) \oplus E^s(v))$ is mapped under composition with the holonomies to an open dense set of the $1$-jets of symplectic maps $P$ as above, so it suffices to check that twisting holds when the map $\psi_{v,w}^{g_0}$ takes value in an open and dense subset of $\text{Sp}(E^u(v) \oplus E^s(v))$.
	
	Again, observe that the condition defining twisting is given by a Zariski open subset of the matrices $\text{Sp}(E^u(v) \oplus E^s(v))$. Hence, as long this set is non-empty the twisting set must also be open and dense in the analytic topology. Then by the paragraph above, this translates to an open and dense condition in $1$-jets of symplectic maps $P$, and as there is no condition imposed on higher jets, we obtain the desired result by Remark \ref{1jets}.
	
   To finish the proof, it thus suffices to check that the Zariski open set defining twisting is non-empty in the symplectic group, which is done below.
	\end{proof}

	\begin{lemma} \label{linalg2}
		There exists a matrix $A \in \text{Sp}(2n)$, where $\R^{2n}$ is taken with standard symplectic basis $\{e_i, f_i\}$ such that $A$ preserves $E^u := \text{span}\, \{e_i\}_{i=1}^n$ and
		$$\forall e^k_I, e^l_{I'}, k + l  = n: (\wedge^k A)(e^k_I) \wedge e^l_{I'} \neq 0. $$	\end{lemma}
	
\begin{proof}
	Note that for fixed $e^k_{I}, e^l_{I'}$ the property that $(\wedge^k A)(e^k_I) \wedge e^l_{I'} \neq 0$ is open in Sp$(2n)$. Thus by induction it suffices to show that or some $e^k_I, e^l_{I'}$  one can arrange so that $(\wedge^k A)(e_I^k) \wedge e^l_{I'}\neq 0$ and moreover $A$ still preserves $E^u$, by an arbitrarily small perturbation of $A \in \text{Sp}(2n)$ -- then repeat inductively by sucessively small perturbations over all pairs $I, I'$.
	
	To prove the claim, suppose $(\wedge^k A)(e^k_I) \wedge e^l_{I'} = 0$, and write $ (\wedge^k A)(e^k_I) = \sum_{J} a_J e^k_J.$ Since $A$ is invertible, there exists $J_0$ such that $a_{J_0} \neq 0$ and such that $|J_0 \cap I'|$ is minimal. Since $|J_0| + |I'| = n$, we have $|J_0 \cap I'| = \{1,\dots, n\}\setminus (J_0 \cup I')$, so we take an arbitrarily chosen bijection $i \mapsto j_i$ from $J_0 \cap I$ to $\{1,\dots, n\}\setminus (J_0 \cup I')$. 
	
	For $\theta > 0$, let $R_\theta^{i,j}$ given by rotating the (oriented) planes $\text{span}(e_i, e_j)$ and $\text{span}(f_i, f_j)$ by $\theta$ and preserving the other basis elements. Let $A'$ be obtained by composing $A$ with each of $R_\theta^{i, j_i}$ for $i \in J_0 \cap I$ (in any order, since the rotation matrices commute). One checks directly that $R_\theta^{i,j} \Omega (R_\theta^{i,j})^T = \Omega$, where $\Omega$ is the standard symplectic form, so $R^{i,j}_\theta$ preserves $E^u$ so $A'$ is symplectic and preserves $E^u$. Writing
	\[\prod_{i \in J_0 \cap I'}(\wedge^k R_\theta^{i, i_j})e_{J}^k = \sum_{L} b_L e^k_L,\]
	by a direct computation one checks that $b_{\{1,..., n\}\setminus I'}\neq 0$ if and only if $J = J_0$, which implies that $(\wedge^k A')(e^k_I) \wedge e^l_{I'} \neq 0$.
\end{proof}

\section{Proof of Theorem \ref{main}} \label{sec: proof}

We finish the proof of the Theorem \ref{main}. In what follows, let $\sigma: \Sigma \to \Sigma$ be the shift map of an invertible subshift of finite type $\Sigma$. The suspension of $\Sigma$ under a continuous $f: \Sigma \to \R^+$ is the compact metric space:
$$\Sigma_f := (\Sigma \times \R)/((x,s) \sim \alpha^n(x,s), \, n \in \Z),$$
where $\alpha(x,s) := (\sigma(x), s - f(x))$. The shift $\sigma$ lifts to a continuous-time system  $\sigma^t_{f}: \Sigma_f \to \Sigma_f$ given by $\sigma^t_f(x, s) = (x, s+ t)$ for $t \in \R$.

First, we need to represent Anosov flows by the suspension of a shift. The following is the standard statement of the construction of a Markov partition for an Anosov flow:

\begin{theorem} \label{markovpart}
	\cite[Theorem 6.6.5]{fh} Let $\Phi: M \to M$ be a $C^1$ Anosov flow. There is a semiconjugacy from a hyperbolic symbolic flow to $\Phi$ that is finite-to-one and one-to-one on a residual set of points, where the roof function for the subshift of finite type corresponds to the travel times between the local sections for the smooth system.	
\end{theorem}

At last we prove Theorem \ref{main}.
 
\begin{proof} [Proof of Theorem 1.1] We prove that the statement holds for all $g \in \G^k_{p,t}$, so the theorem is proved by Proposition \ref{twist}. Fix some such $g_0 \in \G^k_{p,t}$ and let $v, w \in SM$ be the vectors along whose orbits pinching and twisting hold respectively. 
	
	Following the proof of Theorem \ref{markovpart} in \cite{fh}, we see that it is possible to construct the Markov partition so that $v \in SM$ has a unique lift $(p,t)$ to $\Sigma_f$, the suspension of the shift: by enlarging the Markov rectangles by an arbitrarily small amount, one can make the orbit of $v$ only intersect their interior.  Then by \cite[Claim 6.6.9, Corollary 6.6.12]{fh} there is also a unique $(q,s)$ which lifts the homoclinic point with twisting $w$.
	
	Let $P: \Sigma_f \to SM$ be the semi-conjugacy map. We write $\mathcal{E} \to \Sigma_f$ for the pullback of the bundle $E^u \to SM$ to $\Sigma_f$ under $P$, and by $A^t: \mathcal{E} \to \mathcal{E}$ the pullback of the derivative cocycle. By using the return map of $A^t$ to the $0$ section of $\Sigma_f $, the cocycle $A^t$ determines a discrete time cocycle $A$ on $\mathcal{E} \to \Sigma$ identified with $\Sigma \times \{0\} \subseteq \Sigma_f$. Following the propositions in Section 2.1 of \cite{bgv} there exists a distance on $\Sigma$ which makes the cocycle $A$ dominated, so that it admits holonomies. 
	
	First we prove the following lemma which verifies agreement of holonomies of the geodesic flow and its symbolic discrete representation:
	
	\begin{lemma} \label{holsagree} The stable and unstable holonomies $H^{s,u}$ of $A$ on $\mathcal{E}$ are given by the center-stable and unstable $h^{cs,cu}$ holonomies of $\varphi_g^t$ on the stable bundle. 
		
		More precisely, let $x = \bar{x} \times \{0\} \in \Sigma_f^{g_0}$, $\bar{x} \in \Sigma$, and $y = \bar{y} \times \{0\} \in \Sigma_f$, $y \in \Sigma^{g_0}$, where $\bar{y} \in W^s(\bar{x})$, so that $y \in W^{cs}(x)$. Let $v = P(x)$ and $w = P(y)$, then $h^{cs}_{vw} = (H^{s}_{g,g_0})_{\bar{x}\bar{y}}$. The analogous result holds for unstable holonomies. 
	\end{lemma}
	\begin{proof}
		By the proof of existence of holonomies as in \cite{bv}, one obtains the holonomy map as a limit:
		$$H^s_{\bar{x},\bar{y}} =  \lim_{n \to \infty} ((A^{n})_{\bar{x}})^{-1} \circ I_{\sigma^n\bar{x} \sigma^n\bar{y}} \circ (A^{n})_{\bar{y}}.$$
		
		As $n \to \infty$, note that $\sigma^n \bar{x} \times \{0\}$ and $\sigma^n \bar{y} \times \{0\}$ converge to the same stable manifold in $\Sigma_f$. Hence, if we let $T_n := \sum_{i=0}^{n-1} f(\sigma^i\bar{x})$ so that  $(A^{n})_{\bar{x}} = (A^{T_n})_{\bar{x} \times \{0\}}$, then $\sum_{i=0}^{n-1} f(\sigma^i\bar{y}) - (T_n+r) \to 0$, as $n \to \infty$, where $r \in \R$ is such that $\sigma_f^r(y) \in W^s(x)$. 
		
		On the other hand, using the formula defining the holonomies and the definition of $A^t$ as a pullback cocycle of $D\varphi^t_g|_{E^u}$:
		
		$$\begin{aligned} h^{cs}_{vw} &=  \lim_{T \to \infty} (D\varphi_g|_{E^u}^T)_{v}^{-1} \circ I_{\varphi_g^T(v), \varphi_g^{T+r}(w) } \circ (D\varphi_g|_{E^u}^T)_{\varphi_g^r(w)} \circ (D\varphi_g|_{E^u}^r)_{w},  \\
			&=  \lim_{T \to \infty} (A^T)_x^{-1} \circ I_{\sigma^T_f y, \sigma^{T+r}_f y } \circ (A^T)_{\sigma^r_f y} \circ (A^r)_y,
		\end{aligned}$$
		so letting $T = T_n$ we conclude that $ H^s_{\bar{x},\bar{y}} = h^{cs}_{vw}.$ 
	\end{proof}

	Hence, the cocycle $A$ over $\Sigma$ is simple. Let $\rho: SM \to \R$ be a Hölder potential and $\mu_\rho$ its associated equilibrium state for the geodesic flow of $g_0$. Let $\tilde{\rho}$ be the Hölder continuous potential on $\Sigma_{f}$ given by $\tilde{\rho} = \rho \circ P$, and  $\tilde{\mu}_\rho$ its associated equilibrium state for $\sigma_{f}^t: \Sigma_{f} \to \Sigma_{f}$. 
	
	It is a well-known fact (see e.g. \cite{axiomaflows}) that $P_g$ is in fact a measurable isomorphism between $(\Sigma_{f}, \tilde{\mu}_\rho)$ and $(SM, \mu_\rho)$.  Hence the Lyapunov spectrum of $A^t$ with respect to $\tilde{\mu}_\rho$ agrees with that of $D\varphi_g^t$ with respect to $\mu_\rho$, and it suffices to show simplicity of the spectrum of the former.
	
	Since $f: \Sigma \to \R$ is Hölder, identifying $\Sigma$ with $\Sigma \times \{0\} \subseteq \Sigma_{f}$, the Hölder continuous function:
	$$ \int_0^{f(x)} \tilde{\rho}(x, t) \, dt - P(\sigma^t_{f}, \tilde{\rho}) f_g(x),$$
	where $P(\sigma^t_{f}, \tilde{\rho})$ is the pressure of $\sigma^t_{f}$ with respect to $\tilde{\rho}$, defines a potential on $\Sigma = \Sigma \times \{0\} \subseteq \Sigma_{f_g}$ and has a unique equilibrium state $\mu$ which satisfies, for $F \in C^0(\Sigma_{f})$:
	$$ \int_{\Sigma_{f}} F \, d\tilde{\mu}_\rho = \frac{\int_\Sigma \left( \int_0^{f(x)} F(x,t) \, dt \right)d\mu}{\int_\Sigma f(x)\,  d\mu}$$
	by \cite[Proposition 4.3.17]{fh}. In particular, since $\mu$ is an equilibrium state it has local product structure. 
	
	The product $\mu \times dt$ defines a measure for the suspension flow $\sigma_1^t$ on $\Sigma_1$ (where $1$ is the constant function $1$) which has the same Lyapunov spectrum as $\mu$. Since $\mu \times dt$ and $\tilde{\mu}_\rho$  are related by a time change, the Lyapunov spectrum of $A^t$ with respect to $\mu_\rho$ and the Lyapunov spectrum of $A$ with respect to $\mu$ differ by a scalar, see e.g. \cite[Proposition 2.15]{clarklyap}. Hence applying Theorem $\ref{bvmainthm}$ to the simple cocycle $A$ for the measure $\mu$ we obtain simplicity of the Lyapunov spectrum for $\mu_\rho$. 
\end{proof}

\section{Proof of Theorems \ref{volmain} and \ref{generalmain}} \label{sec:volpres}

In this section we explain the needed modifications to the previous sections to give the proofs of Theorems \ref{volmain} and \ref{generalmain}:

\begin{proof}[Proof of Theorems \ref{volmain} and \ref{generalmain}] For $\frac{1}{2}$-bunched Anosov flows, the splitting $E^u \oplus E^0 \oplus E^s$ may not be $C^1$, so instead we consider the derivative cocycle on the $C^1$-bundles $Q^u := E^{cu}/E^0$ and $Q^s:= E^{cs}/E^0$, which we have shown to be $1$-bunched in Proposition \ref{bunchingquotient}. In what follows, we prove simplicity for the spectrum on $Q^u$ and $Q^s$ implies the desired result since $D\Phi$ on $Q^{u,s}$ has the same spectrum as $D\Phi$ on $E^{u,s}$.

	For Theorem \ref{generalmain}, recall that topological mixing is $C^1$-open and $C^k$-dense in the space of Anosov flows. Then we follow the propositions in Section \ref{sec: pt} to construct orbits with pinching and twisting for the cocycle on $Q^u$ by a $C^k$-small perturbation, which in this case is achievable since the analogue of Theorem \ref{ktmain} is clear in the space of all vector fields and $\mathfrak{X}^k_A(X)$ is open by structural stability in the space of all vector fields and moreover the linear algebra lemmas (Lemma \ref{linalg}, Lemma \ref{linalg2}) needed for the case of Sp$(2n)$ are immediate for GL$(n)$. The $C^1$-openness of the conditions also is proved similarly. Then by a symmertric argument it is clear that pinching and twisting for both $Q^u$ and $Q^s$ is $C^1$-open and $C^k$-dense. The proof then follows the same outline in Section \ref{sec: proof}.

	The proof of Theorem \ref{volmain} is similar, in that the linear algebra lemmas (Lemma \ref{linalg}, Lemma \ref{linalg2}) needed for the case of Sp$(2n)$ are still  immediate for SL$(n)$. Moreover, topological mixing is known for all $C^2$-volume-preserving Anosov flows. Finally, it remains to prove an analogue of Theorem \ref{ktmain} for the conservative class, which we do in the next section. With that in hand, the proof also follows the same outline as Theorem \ref{main}.
\end{proof}

\subsection{Conservative Perturbations} In this section we prove the analogue of Theorem \ref{ktmain} in the volume-preserving category. To the best of the author's knowledge the result is not found anywhere in the literature so the complete proof is included here. Throughout, we let $X \in \mathfrak{X}_{m}^\infty(M)$ be a non-vanishing vector field generating the flow $\varphi_X$ on the smooth manifold $M$ which preserves the smooth volume $m$. Fix an embedded segment of a flow orbit $l:[0,\eps] \to M$ parametrized by the time-parameter and a small transversal smooth hypersurface  $\Sigma(0)$ to $X$ at $l(0)$. 

For $t \in [0,\eps]$, set $\Sigma(t) = \varphi_X^t(\Sigma(0))$ so that $\iota_X m$ is a volume form on the hypersurfaces $\Sigma(t)$. The following result, whose proof is elementary except for an application of the conservative pasting lemma, shows that it is possible to perturb the $k$-jets in the conservative setting generically by $C^k$-small perturbations.

\begin{theorem} \label{conspres} Let $Q$ be some dense subset of the space of $k$-jets of volume-preserving maps $(\Sigma(0), \iota_X m,l(0)) \to (\Sigma(\eps),  \iota_X m, l(\eps))$.
	
	Then there is arbitrarily $C^k$-close to $X$ an $m$-preserving $X'$ such that:
	
	\begin{enumerate} \item [(a)] $Y:= X' -X$ is supported in an arbitrarily small tubular neighborhood $B$ of $l([\delta, \eps-\delta])$, for some $0< \delta < \eps$;
		\item [(b)] $Y = 0$ on $l([0,1])$ and $Y$ is tangent to the hypersurfaces $\Sigma(t)$;
		\item [(c)] The flow of $X'$ generates a map $(\Sigma(0), l(0)) \to (\Sigma(\eps), l(\eps))$ with $k$-jets in $Q$.
		
	\end{enumerate} 
\end{theorem}

\begin{proof} If $B$ is sufficiently small we may assume that it is foliated by the transversals $\Sigma(t)$ and, moreover, by passing to a further neighborhood we may assume that the transverse sections are mapped diffeomorphically onto each other by the flow $X$, i.e., we may construct the perturbation in a flow box with transversals given by the $\Sigma(t)$. 
	
In the flowbox, a classic application of Moser's trick allows us to assume that the flow is in normal coordinates $\varphi^t_X (x_1, ..., x_{n-1}, s) \mapsto (x_1, ..., x_{n-1}, s +t)$, where the image of $l$ is contained in $\{x_1 = ... = x_n = 0\}$ and $m = dx^1 \wedge ... \wedge dx^n$ . In these coordinates, we may regard $B \cong U \times [0,\eps]$, where $U\subseteq \R^{n-1}$ is a domain and so $Q \subseteq J^k_{m}(n-1, \R)$, where $J^k_{m}(n-1, \R)$ is the Lie group of $k$-jets of volume preserving maps fixing the origin. 

Using the flow to identify the fibers of $U \times \R \to \R$, the problem is thus reduced to the construction, for each $\delta > 0$, of a time dependent vector field $\{Y_t\}_{t \in [0,\eps]}$ on $\R^{n-1}$ with the following properties:
\begin{enumerate}
	\item [(a)] $Y_t(0) = 0$ and $\text{supp}(Y_t) \subseteq U$ for $t \in [0,\eps]$;	
	\item [(b)] $Y_t \equiv 0$ on $[0, \delta]$ and $Y_t \equiv 0$ on $[\eps-\delta, \eps]$;
	\item [(c)] $Y_t$ is divergence free for all $t \in [0,\eps]$;
	\item [(d)] The time-$\eps$ map $f: (\R^{n-1},0) \to (\R^{n-1},0)$ of $Y_t$ has derivative at $0$ in $Q$;
	\item [(e)] $\|Y_t\|_{C^\infty} < \delta$ for all $t \in [0,\eps];$
\end{enumerate}

The construction is given by first specifying the time-$\eps$ map $f$ and then finding an appropriate isotopy within the volume-preserving category to the identity. 

Fix some $\theta \in Q$ sufficiently close to the $k$-jets of $I$ (the identity map) and a map $F: (\R^{n-1},0) \to (\R^{n-1},0)$ whose $k$-jet at the origin is given by $\theta$. Take some $C^\infty$ bump function $\rho: \R^{n-1} \to \R$ which interpolates between the constant function $1$ in $B(0, \eta/2)$ to the constant function $0$ outside of $B(0,\eta)$ for some $\eta$ small. Let $F' = \rho F$; if $\theta$ is sufficiently close to $0$, then $||F'-I||_{C^k}$ is small so in particular $F'\in \text{Diff}(\R^{n-1})$. Applying Moser's trick, we can find an $f \in \text{Diff}_{m}(\R^{n-1})$, i.e. preserving $m$, which is $C^k$-close to the identity and which agrees with $F'$ where it is conservative, namely, everywhere except $B(0,\eta)\setminus B(0, \eta/2)$. In particular, the $k$-jet of $f$ at the origin equals $\theta \in Q$.

To obtain such an $f$, we construct a family $s \mapsto h_{s} \in \text{Diff}(\R^{n-1})$ such that $h_1 = F$ and $h_0 =:f$ is conservative. Let $r: B \to \R$ be the smooth function $C^k$ close to $1$ satisfying $F_*\mu = r\mu$. Then for $s \in [0,1]$ we solve $(h_s)_* \mu = r^s \mu$, namely div$Z_{s} =  r^s\log r^s$, where $Z_{s} = \partial_s h_{s}$. Moreover, the proof of the Poincaré Lemma shows that we can take $Z_s$ to be constant equal to $0$ outside of of $B(0,\eta)$. By the conservative pasting lemma \cite{pasting}, there exists $W_s$ which agrees with $Z_s$ on a neighborhood of $\R^{n-1} \setminus (B(0,\eta)\setminus B(0, \eta/2))$ and is divergence-free. Then $Z'_s = Z_s - W_s$ also satisfies div$Z'_{s} =  r^s\log r^s$ and it is identically $0$ where $r = 1$, so that $h_s(x) = F(x)$ on $B(0,\eta)\setminus B(0, \eta/2)$, where now $\partial_s h_s = Z'_s$. In particular, $f:= h_0$ is the identity outside of $B(0,\eta)$ and its $k$-jet at the origin is given by $\theta$. This constructs the desired $f$.

Now let $\alpha:[0,\eps]\to [0,\eps]$ be a $C^\infty$ function such that
$\alpha \equiv 0$ on $[0, \delta]$ and $\alpha \equiv \eps$ on $[\eps-\delta, \eps]$. If $\|f-I\|_{C^k}$ is sufficiently small (which is ensured by taking $\theta$ closer to the jets of the identity), the maps $g_t := \alpha(t) f + (1-\alpha(t))I$ are all diffeomorphisms and $t \mapsto \partial_t g_t$ is a time-dependent vector field that satisfies all desired properties except for being divergence-free. 

To repair that, again Moser's trick constructs a family $s \mapsto g_{t,s}$ such that $g_{t,0} = g_t$ and $g_{t,1}$ is conservative as follows. Let $r_t: B \to \R$ be the smooth 1-parameter family of smooth functions $C^k$ close to $1$ satisfying $(g_t)_*\mu = r_t\mu$. Then for $s \in [0,1]$ we solve $(g_{t,s})_* \mu = r_t^s \mu$, namely div$Z_{t,s} =  r^s_t\log r_t^s$, where $Z_{t,s} = \partial_s g_{t,s}$. It is an easy consequence of the proof of the Poincaré lemma that the family $Z_{t,s}$ may be taken to be smooth in $t$ with small $t$ derivatives, since $t \mapsto r_t$ as a 1-parameter family has the same properties. Moreover, we can take supp $Z_{t,s} \subseteq \text{supp} \, (r_t-1)$. The $C^k$ norm of the $Z_{t,s}$ is a continuous function of the $C^k$ norm of $r^s_t\log r_t^s$, so that taking $Y_t = \partial_t g_{t,1}$ finishes the proof.
\end{proof}

\end{document}